\documentclass[preprint,12pt]{elsarticle}

\usepackage{amssymb}
\usepackage{amsmath}
\newtheorem{theorem}{Theorem}

\newtheorem{remark}[theorem]{Remark}

\newenvironment{proof}[1][Proof]{\noindent\textbf{#1.} }{\ \rule{0.5em}{0.5em}}

\journal{Journal of Mathematical Analysis and Applications}

\begin{document}

\begin{frontmatter}



\title{Closed-form expressions for derivatives of Bessel functions with
respect to the order}


\author{J. L. Gonz\'{a}lez-Santander}

\address{C/ Ovidi Montllor i Mengual 7, pta. 9 \\
46017 Valencia, Spain \\
juanluis.gonzalezsantander@gmail.com}

\begin{abstract}
We have used recent integral representations of the derivatives of the Bessel functions with respect to the order to obtain closed-form expressions in terms of generalized hypergeometric functions and Meijer-$G$ functions.
Also, we have carried out similar calculations for the derivatives of the modified Bessel functions with respect to the order, obtaining closed-form expressions as well. For this purpose, we have obtained integral representations of the 
derivatives of the modified Bessel functions with respect to the order. As by-products, we have calculated two non-tabulated integrals.
\end{abstract}


\end{frontmatter}


\section{Introduction}

The Bessel functions have had many applications since F. W. Bessel
(1784-1846) found this kind of functions in his studies of planetary motion.
In Physics, these functions arise naturally in boundary value problems
of potential theory for cylindrical domains \cite[Chap.6]{Lebedev}. In
Mathematics, the Bessel functions are encountered in the theory of
differential equations with turning points, and well as with poles \cite[%
Sect. 10.72]{DLMF}. Thereby, the theory of Bessel functions has been studied
extensively in many classical textbooks \cite{Watson, Andrews}.

Usually, the definition of the Bessel function of the first kind $J_{\nu
}\left( z\right) $ and the modified Bessel function $I_{\nu }\left( z\right) 
$ are given in series form as follows:\ 
\begin{equation}
J_{\nu }\left( z\right) =\left( \frac{z}{2}\right) ^{\nu }\sum_{k=0}^{\infty
}\frac{\left( -1\right) ^{k}\left( z/2\right) ^{2k}}{k!\, \Gamma \left( \nu
+k+1\right) },  \label{Jnu_def}
\end{equation}%
and%
\begin{equation}
I_{\nu }\left( z\right) =\left( \frac{z}{2}\right) ^{\nu }\sum_{k=0}^{\infty
}\frac{\left( z/2\right) ^{2k}}{k!\, \Gamma \left( \nu +k+1\right) }.
\label{Inu_def}
\end{equation}

The Bessel function of the second kind $Y_{\nu }\left( z\right) $ is defined
in terms of the Bessel function of the first kind as 
\begin{equation}
Y_{\nu }\left( z\right) =\frac{J_{\nu }\left( z\right) \cos \pi \nu -J_{-\nu
}\left( z\right) }{\sin \pi \nu },\qquad \nu \notin 
\mathbb{Z}
,  \label{Ynu_def}
\end{equation}%
and similarly, for the Macdonald function $K_{\nu }\left( z\right) $, we have%
\begin{equation}
K_{\nu }\left( z\right) =\frac{\pi }{2}\frac{I_{-\nu }\left( z\right)
-I_{\nu }\left( z\right) }{\sin \pi \nu },\qquad \nu \notin 
\mathbb{Z}
.  \label{Knu_def}
\end{equation}

Despite the fact, that the literature about the Bessel functions is very
large as mentioned before, the literature regarding the derivatives of
 $J_{\nu }$, $Y_{\nu }$, $I_{\nu }$ and $K_{\nu }$ with
respect to the order $\nu $ is relatively scarce. For instance, for $\nu
=\pm 1/2$ we find expressions for the order derivatives in terms of the
exponential integral $\mathrm{Ei}\left( z\right) $ and the sine and cosine
integrals, $\mathrm{Ci}\left( z\right) $ and $\mathrm{Si}\left( z\right) $ 
\cite{Oberhettiger, Brychov}. By using the recurrence relations of Bessel
 \cite[Eqn. 10.6.1]{DLMF} and modified Bessel functions \cite[Eqn.
5.7.9]{Lebedev}, we can derive expressions for half-integral order $\nu
=n\pm 1/2$. Also, for integral order $\nu =n$ we find some series
representations in \cite{Brychov}. For arbitrary order, we have the
following series representations \cite[Eqns. 10.15.1 \& 10.38.1]{DLMF}%
\begin{equation}
\frac{\partial J_{\nu }\left( z\right) }{\partial \nu }=J_{\nu }\left(
z\right) \log \left( \frac{z}{2}\right) -\left( \frac{z}{2}\right) ^{\nu
}\sum_{k=0}^{\infty }\frac{\psi \left( \nu +k+1\right) \left( -1\right)
^{k}\left( z/2\right) ^{2k}}{k!\, \Gamma \left( \nu +k+1\right) },
\label{DJnu_series}
\end{equation}%
and%
\begin{equation}
\frac{\partial I_{\nu }\left( z\right) }{\partial \nu }=I_{\nu }\left(
z\right) \log \left( \frac{z}{2}\right) -\left( \frac{z}{2}\right) ^{\nu
}\sum_{k=0}^{\infty }\frac{\psi \left( \nu +k+1\right) \left( z/2\right)
^{2k}}{k!\, \Gamma \left( \nu +k+1\right) },  \label{DInu_series}
\end{equation}%
which are obtained directly from (\ref{Jnu_def})\ and (\ref{Inu_def}). Also,
from (\ref{Ynu_def})\ and (\ref{Knu_def}), we can calculate the order
derivative of $Y_{\nu }$ and $K_{\nu }$ as \cite[Eqns 10.15.2 \& 10.38.2]%
{DLMF}, 
\begin{equation}
\frac{\partial Y_{\nu }\left( z\right) }{\partial \nu }=\cot \pi \nu \left[ 
\frac{\partial J_{\nu }\left( z\right) }{\partial \nu }-\pi \ Y_{\nu }\left(
z\right) \right] -\csc \pi \nu \,\frac{\partial J_{-\nu }\left( z\right) }{%
\partial \nu }-\pi \ J_{\nu }\left( z\right) ,  \label{DYnu_(DJnu)}
\end{equation}%
and%
\begin{equation}
\frac{\partial K_{\nu }\left( z\right) }{\partial \nu }=\frac{\pi }{2}\csc
\pi \nu \left[ \frac{\partial I_{-\nu }\left( z\right) }{\partial \nu }-%
\frac{\partial I_{\nu }\left( z\right) }{\partial \nu }\right] -\pi \cot \pi
\nu \,K_{\nu }\left( z\right) .  \label{DKnu_(DInu)}
\end{equation}

Although we can accelerate the convergence of the alternating series given
in (\ref{DJnu_series}) by using Cohen-Villegas-Zagier algorithm \cite%
{CohenVillegasZagier}, this series does not converge properly for large $z$,
and it is not useful from a numeric point of view. Also, the series given in
(\ref{DInu_series})\ is not useful for large $z$ as well.

Nonetheless, in the literature we find integral representations of $J_{\nu
}\left( z\right) $ and $I_{\nu }\left( z\right) $ in \cite{Apelblat}, which
read as,%
\begin{equation}
\frac{\partial J_{\nu }\left( z\right) }{\partial \nu }=\pi \nu
\int_{0}^{\pi /2}\tan \theta \ Y_{0}\left( z\sin ^{2}\theta \right) J_{\nu
}\left( z\cos ^{2}\theta \right) d\theta ,\qquad \mathrm{Re}\,\nu >0,
\label{DJnu_int_Apelblat}
\end{equation}%
and%
\begin{equation}
\frac{\partial I_{\nu }\left( z\right) }{\partial \nu }=-2\nu \int_{0}^{\pi
/2}\tan \theta \ K_{0}\left( z\sin ^{2}\theta \right) I_{\nu }\left( z\cos
^{2}\theta \right) d\theta ,\qquad \mathrm{Re}\,\nu >0.
\label{DInu_int_Apelblat}
\end{equation}

Recently, new integral representations of $J_{\nu }\left( z\right) $ and $%
Y_{\nu }\left( z\right) $ are given in \cite{Dunster} for $\nu >0$, $%
\left\vert \mathrm{arg}\ z\right\vert \leq \pi $, and $z\neq 0$:%
\begin{equation}
\frac{\partial J_{\nu }\left( z\right) }{\partial \nu }=\pi \nu \left[
Y_{\nu }\left( z\right) \int_{0}^{z}\frac{J_{\nu }^{2}\left( t\right) }{t}%
dt+J_{\nu }\left( z\right) \int_{z}^{\infty }\frac{J_{\nu }\left( t\right)
Y_{\nu }\left( t\right) }{t}dt\right] ,  \label{DJnu_int_Dunster}
\end{equation}%
and%
\begin{eqnarray}
&&\frac{\partial Y_{\nu }\left( z\right) }{\partial \nu }
\label{DYnu_int_Dunster} \\
&=&\pi \nu \left[ J_{\nu }\left( z\right) \left( \int_{z}^{\infty }\frac{%
Y_{\nu }^{2}\left( t\right) }{t}dt-\frac{1}{2\nu }\right) -Y_{\nu }\left(
z\right) \int_{z}^{\infty }\frac{J_{\nu }\left( t\right) Y_{\nu }\left(
t\right) }{t}dt\right] .  \notag
\end{eqnarray}

It is worth noting that \cite{Dunster} does not state the following direct
result from (\ref{DJnu_int_Dunster})\ and (\ref{DYnu_int_Dunster}), 
\begin{eqnarray}
&&\frac{\partial }{\partial \nu }\left( J_{\nu }\left( z\right) Y_{\nu
}\left( z\right) \right)  \label{D(Jnu*Ynu)_int} \\
&=&\pi \nu \left[ Y_{\nu }^{2}\left( z\right) \int_{0}^{z}\frac{J_{\nu
}^{2}\left( t\right) }{t}dt+J_{\nu }^{2}\left( z\right) \left(
\int_{z}^{\infty }\frac{Y_{\nu }^{2}\left( t\right) }{t}dt-\frac{1}{2\nu }%
\right) \right] .  \notag
\end{eqnarray}

However, in \cite{BrychovNew}, we found $\partial J_{\nu }/\partial \nu $\ in
closed-form as,%
\begin{eqnarray}
&&\frac{\partial J_{\nu }\left( z\right) }{\partial \nu }
\label{D_J_nu_Brychov} \\
&=&\frac{\pi \left[ Y_{\nu }\left( z\right) -\cot \pi \nu \ J_{\nu }\left(
z\right) \right] }{2\, \Gamma ^{2}\left( \nu +1\right) }\left( \frac{z}{2}%
\right) ^{2\nu }\,_{2}F_{3}\left( \left. 
\begin{array}{c}
\nu ,\nu +\frac{1}{2} \\ 
\nu +1,\nu +1,2\nu +1%
\end{array}%
\right\vert -z^{2}\right)  \notag \\
&&+J_{\nu }\left( z\right) \left[ \frac{1}{2\nu }-\psi \left( \nu +1\right)
+\log \left( \frac{z}{2}\right) 
\begin{array}{c}
\displaystyle
\\ 
\displaystyle%
\end{array}%
\right.  \notag \\
&&\quad +\left. \frac{z^{2}}{4\left( \nu ^{2}-1\right) }\,_{3}F_{4}\left(
\left. 
\begin{array}{c}
1,1,\frac{3}{2} \\ 
2,2,2-\nu ,2+\nu%
\end{array}%
\right\vert -z^{2}\right) \right] .  \notag
\end{eqnarray}%
Also, from (\ref{D_J_nu_Brychov}), the derivatives of $Y_{\nu }\left( z\right) 
$, $I_{\nu }\left( z\right) $ and $K_{\nu }\left( z\right) $ with respect to
the order are calculated. Nevertheless, in \cite{BrychovNew}, the calculation of (%
\ref{D_J_nu_Brychov})\ relies mostly on symbolic computer algebra. Since
this calculation is highly non-trivial, the aim of this paper is to
provide such calculation. For this purpose, we calculate the integrals given
in (\ref{DJnu_int_Dunster}) and (\ref{DYnu_int_Dunster}). Moreover, we derive integral
representations similar to (\ref{DJnu_int_Dunster})\ and (\ref%
{DYnu_int_Dunster})\ for the modified Bessel functions $%
I_{\nu }$ and $K_{\nu }$, wherein the integrals are calculated in
closed-form as well. Therefore, the scope of this paper is the
calculation of all these integrals to justify the closed-form
expressions of the order derivatives of the Bessel and modified Bessel
functions found in the literature. Also, since these closed-form expressions cannot be applied for $\nu \in 
\mathbb{Z}
$, alternative expressions in terms of Meijer-$G$\ functions are calculated
as well.

We organize this article as follows. In Section \ref{Section: Bessel
functions}, we calculate the integrals that appear in (\ref{DJnu_int_Dunster})\
and (\ref{DYnu_int_Dunster}). For this purpose, we introduce the generalized
hypergeometric function and its asymptotic behavior to rewrite (\ref%
{DJnu_int_Dunster})-(\ref{D(Jnu*Ynu)_int})\ in closed-form. In Section \ref%
{Section: Modified Bessel functions},\  we derive integral representations of $\partial
I_{\nu }/\partial \nu $ and $\partial K_{\nu }/\partial \nu $\ similar to (%
\ref{DJnu_int_Dunster}) and (\ref{DYnu_int_Dunster}). We calculate the integrals of these integral representations to obtain the derivatives of the modified Bessel functions with respect to the order in closed-form. Section \ref{Section:
Meijer_G}\ is devoted to these same calculations, but in terms of Meijer-$G$
functions. Finally, we collect the conclusions in Section \ref{Section:
Conclusions}.

\section{Order derivatives for Bessel functions\label{Section: Bessel
functions}}

To calculate the integrals given in (\ref{DJnu_int_Dunster}) and (%
\ref{DYnu_int_Dunster}),\ we have to introduce the generalized
hypergeometric function:\ 
\begin{equation}
_{p}F_{q}\left( \left. 
\begin{array}{c}
a_{1},\ldots ,a_{p} \\ 
b_{1},\ldots ,b_{q}%
\end{array}%
\right\vert z\right) =\sum_{k=0}^{\infty }\frac{\left( a_{1}\right)
_{k}\cdots \left( a_{p}\right) _{k}}{\left( b_{1}\right) _{k}\cdots \left(
b_{q}\right) _{k}}\frac{z^{k}}{k!},  \label{p_F_q_def}
\end{equation}%
where $\left( \alpha \right) _{k}$ denotes the Pochhammer symbol \cite[Eqn.
5.2.5]{DLMF}, 
\begin{equation}
\left( \alpha \right) _{k}=\frac{\Gamma \left( \alpha +k\right) }{\Gamma
\left( \alpha \right) }.  \label{Pochhamer_def}
\end{equation}

Next, we present an equivalent way to define a hypergeometric function \cite[%
Sect. 2.1]{Andrews}. Any series 
\begin{equation*}
\sum_{k=0}^{\infty }c_{k},
\end{equation*}%
that satisfies 
\begin{equation}
\frac{c_{k+1}}{c_{k}}=\frac{\left( k+a_{1}\right) \cdots \left(
k+a_{p}\right) z}{\left( k+1\right) \left( k+b_{1}\right) \cdots \left(
k+b_{q}\right) },  \label{c_k+1/c_k}
\end{equation}%
defines a hypergeometric series 
\begin{equation}
\sum_{k=0}^{\infty }c_{k}=c_{0}\,_{p}F_{q}\left( \left. 
\begin{array}{c}
a_{1},\ldots ,a_{p} \\ 
b_{1},\ldots ,b_{q}%
\end{array}%
\right\vert z\right) .  \label{p_F_q_c_k}
\end{equation}

The first integral of (\ref{DJnu_int_Dunster})\ can be calculated
straightforwardly from the following tabulated integral \cite[Eqn. 1.8.3]%
{Prudnikov2}:%
\begin{eqnarray}
&&\int_{0}^{x}t^{\lambda }J_{\nu }\left( t\right) J_{\mu }\left( t\right) dt=%
\frac{x^{\lambda +\mu +\nu +1}}{2^{\mu +\nu }\left( \lambda +\mu +\nu
+1\right) \Gamma \left( \mu +1\right) \Gamma \left( \nu +1\right) }
\label{Int_Prudnikov} \\
&&\qquad \qquad \qquad \qquad \times \,_{3}F_{4}\left( \left. 
\begin{array}{c}
\frac{\mu +\nu +1}{2},\frac{\mu +\nu +2}{2},\frac{\lambda +\mu +\nu +1}{2}
\\ 
\mu +1,\nu +1,\mu +\nu +1,\frac{\lambda +\mu +\nu +3}{2}%
\end{array}%
\right\vert -x^{2}\right)  \notag \\
&&\mathrm{Re}\left( \lambda +\mu +\nu \right) >-1,  \notag
\end{eqnarray}%
thus, if $\nu >0$, we have%
\begin{equation}
\int_{0}^{z}\frac{J_{\nu }^{2}\left( t\right) }{t}dt=\frac{\left( z/2\right)
^{2\nu\ }}{2\nu \, \Gamma ^{2}\left( \nu +1\right) }\,_{2}F_{3}\left( \left. 
\begin{array}{c}
\nu ,\nu +\frac{1}{2} \\ 
\nu +1,\nu +1,2\nu +1%
\end{array}%
\right\vert -z^{2}\right) .  \label{Int_Jnu^2b}
\end{equation}

The second integral in (\ref{DJnu_int_Dunster})\ is calculated as follows.

\begin{theorem}
If $z\neq 0$, $\left\vert \mathrm{arg}\ z\right\vert <\pi $ and $\nu >0$, $%
\nu \notin 
\mathbb{Z}
$, the following integral holds true: 
\begin{eqnarray}
&&\int_{z}^{\infty }\frac{J_{\nu }\left( t\right) Y_{\nu }\left( t\right) }{t%
}dt  \label{Int_Jnu_Ynu} \\
&=&\frac{-1}{\pi \nu }\left[ \log \left( \frac{2}{z}\right) +\psi \left( \nu
\right) +\frac{1}{2\nu }%
\begin{array}{c}
\displaystyle
\\ 
\displaystyle%
\end{array}%
\right.  \notag \\
&&+\frac{\pi \cot \pi \nu \ \left( z/2\right) ^{2\nu }}{2\, \Gamma ^{2}\left(
\nu +1\right) }\,_{2}F_{3}\left( \left. 
\begin{array}{c}
\nu ,\nu +\frac{1}{2} \\ 
\nu +1,\nu +1,2\nu +1%
\end{array}%
\right\vert -z^{2}\right)  \notag \\
&&\quad \left. +\frac{z^{2}}{4\left( 1-\nu ^{2}\right) }\,_{3}F_{4}\left(
\left. 
\begin{array}{c}
1,1,\frac{3}{2} \\ 
2,2,2-\nu ,2+\nu%
\end{array}%
\right\vert -z^{2}\right) \right] .  \notag
\end{eqnarray}
\end{theorem}

\begin{proof}
First, let us calculate the corresponding indefinite integral of (\ref%
{Int_Jnu_Ynu}) applying the definition of the $Y_{\nu }\left( z\right) $
function given in (\ref{Ynu_def}). Thereby, we have 
\begin{equation}
\int \frac{J_{\nu }\left( t\right) Y_{\nu }\left( t\right) }{t}dt=\cot \pi
\nu \int \frac{J_{\nu }^{2}\left( t\right) }{t}dt-\csc \pi \nu \int \frac{%
J_{-\nu }\left( t\right) J_{\nu }\left( t\right) }{t}dt.
\label{Int_Jnu_Ynu_indefinida_1}
\end{equation}%
Notice that the first integral on the RHS\ of (\ref{Int_Jnu_Ynu_indefinida_1}%
)\ has been calculated in (\ref{Int_Jnu^2b}). However, the general
expression given in (\ref{Int_Prudnikov})\ fails for the second integral.
Nonetheless, taking $\mu =-\nu $ in the following expression \cite[Eqn.
10.8.3]{DLMF}%
\begin{equation*}
J_{\nu }\left( z\right) J_{\mu }\left( z\right) =\left( \frac{z}{2}\right)
^{\mu +\nu }\sum_{n=0}^{\infty }\frac{\left( \nu +\mu +n+1\right) _{n}\left(
-1\right) ^{n}\left( z/2\right) ^{2n}}{n!\, \Gamma \left( n+\mu +1\right)
\Gamma \left( n+\nu +1\right) },
\end{equation*}%
and separating the first term, we can integrate term by term, arriving at 
\begin{eqnarray}
&&\int \frac{J_{-\nu }\left( t\right) J_{\nu }\left( t\right) }{t}dt
\label{Int_Jnu_J-nu} \\
&=&\frac{\log t}{\Gamma \left( 1+\nu \right) \Gamma \left( 1-\nu \right) }+%
\frac{1}{2}\sum_{k=1}^{\infty }\frac{\Gamma \left( 2k+1\right) \left(
-1\right) ^{k}\left( t/2\right) ^{2k+1}}{k!\ k\, \Gamma \left( k+1\right)
\Gamma \left( k+\nu +1\right) \Gamma \left( k-\nu +1\right) },  \notag
\end{eqnarray}%
where we have used the definition of the Pochhammer symbol (\ref%
{Pochhamer_def}). Using now the following properties of the gamma function 
\cite[Eqn. 1.2.1\&2]{Lebedev}:%
\begin{equation}
\Gamma \left( z+1\right) =z\, \Gamma \left( z\right) ,  \label{Gamma_factorial}
\end{equation}%
and 
\begin{equation}
\Gamma \left( z\right) \Gamma \left( 1-z\right) =\frac{\pi }{\sin \pi z}%
,\quad z\notin 
\mathbb{Z}
,  \label{Gamma_reflection}
\end{equation}%
and expressing the sum given in (\ref{Int_Jnu_J-nu})\ as a hypergeometric
function, after some simplification, we arrive at%
\begin{eqnarray}
&&\int \frac{J_{-\nu }\left( t\right) J_{\nu }\left( t\right) }{t}dt
\label{Int_Jnu_J-nu_resultado} \\
&=&\frac{\sin \pi \nu }{\pi \nu }\left\{ \log t-\frac{t^{2}}{4\left( 1-\nu
^{2}\right) }\,_{3}F_{4}\left( \left. 
\begin{array}{c}
1,1,\frac{3}{2} \\ 
2,2,2-\nu ,2+\nu%
\end{array}%
\right\vert -t^{2}\right) \right\} .  \notag
\end{eqnarray}%
Inserting now the results (\ref{Int_Jnu^2b})\ and (\ref%
{Int_Jnu_J-nu_resultado})\ in (\ref{Int_Jnu_Ynu_indefinida_1}), we obtain%
\begin{eqnarray}
&&\int \frac{J_{\nu }\left( t\right) Y_{\nu }\left( t\right) }{t}dt
\label{Int_Jnu_Ynu_indefinida} \\
&=&\frac{1}{\pi \nu }\left[ -\log t+\left( \frac{t}{2}\right) ^{2\nu }\frac{%
\pi \cot \pi \nu }{2\, \Gamma ^{2}\left( \nu +1\right) }\,_{2}F_{3}\left(
\left. 
\begin{array}{c}
\nu ,\nu +\frac{1}{2} \\ 
\nu +1,\nu +1,2\nu +1%
\end{array}%
\right\vert -t^{2}\right) \right.  \notag \\
&&\qquad +\left. \frac{t^{2}}{4\left( 1-\nu ^{2}\right) }\,_{3}F_{4}\left(
\left. 
\begin{array}{c}
1,1,\frac{3}{2} \\ 
2,2,2-\nu ,2+\nu%
\end{array}%
\right\vert -t^{2}\right) \right] .  \notag
\end{eqnarray}%
To calculate (\ref{Int_Jnu_Ynu_indefinida})\ with the integration
limits given in (\ref{Int_Jnu_Ynu}), we have to perform the following
limits:%
\begin{equation}
\lim_{t\rightarrow \infty }\, \frac{\cot \pi \nu }{2\nu\, \Gamma ^{2}\left( \nu
+1\right) }\,\left( \frac{t}{2}\right) ^{2\nu }\,_{2}F_{3}\left( \left. 
\begin{array}{c}
\nu ,\nu +\frac{1}{2} \\ 
\nu +1,\nu +1,2\nu +1%
\end{array}%
\right\vert -t^{2}\right) ,  \label{lim_2F3}
\end{equation}%
and%
\begin{equation}
\lim_{t\rightarrow \infty }\, \frac{t^{2}}{4\pi \nu \left( 1-\nu ^{2}\right) }%
\,_{3}F_{4}\left( \left. 
\begin{array}{c}
1,1,\frac{3}{2} \\ 
2,2,2-\nu ,2+\nu%
\end{array}%
\right\vert -t^{2}\right) .  \label{lim_3F4}
\end{equation}%
For this purpose, let us apply the following asymptotic formula for the $%
_{p}F_{p+1}$\ hypergeometric function as $\left\vert z\right\vert
\rightarrow \infty $ (see \cite[Sect. 16.11]{DLMF}): 
\begin{eqnarray}
&&_{p}F_{p+1}\left( \left. 
\begin{array}{c}
a_{1},\ldots ,a_{p} \\ 
b_{1},\ldots b_{p+1}%
\end{array}%
\right\vert z\right)  \label{pFp+1_infinito} \\
&=&\frac{\prod_{j=1}^{p+1}\Gamma \left( b_{j}\right) }{\sqrt{\pi }%
\prod_{k=1}^{p}\Gamma \left( a_{k}\right) }\left( -z\right) ^{\chi }\left\{
\cos \left( \pi \chi +2\sqrt{-z}\right) \left[ 1+O\left( \frac{1}{z}\right) %
\right] \right.  \notag \\
&&+\left. \frac{c_{1}}{2\sqrt{-z}}\sin \left( \pi \chi +2\sqrt{-z}\right) %
\left[ 1+O\left( \frac{1}{z}\right) \right] \right\}  \notag \\
&&+\frac{\prod_{j=1}^{p+1}\Gamma \left( b_{j}\right) }{\prod_{k=1}^{p}\Gamma
\left( a_{k}\right) }\sum_{k=1}^{p}\frac{\Gamma \left( a_{k}\right)
\prod_{j=1,j\neq k}^{p}\Gamma \left( a_{j}-a_{k}\right) }{%
\prod_{j=1}^{p+1}\Gamma \left( b_{j}-a_{k}\right) }\left( -z\right) ^{-a_{k}}%
\left[ 1+O\left( \frac{1}{z}\right) \right] ,  \notag
\end{eqnarray}%
wherein the case of simple poles (i.e. $a_{j}-a_{k}\notin 
\mathbb{Z}
$) and the following definitions are considered: 
\begin{eqnarray*}
A_{p} &=&\sum_{k=1}^{p}a_{k},\qquad B_{p+1}=\sum_{k=1}^{p+1}b_{k}, \\
\chi &=&\frac{1}{2}\left( A_{p}-B_{p+1}+\frac{1}{2}\right) , \\
\mathbf{A} &=&\sum_{s=2}^{p}\sum_{j=1}^{s-1}a_{s}a_{j},\qquad \mathbf{B}%
=\sum_{s=2}^{p+1}\sum_{j=1}^{s-1}b_{s}b_{j}, \\
c_{1} &=&2\left( \mathbf{B}-\mathbf{A}+\frac{1}{4}\left(
3A_{p}+B_{p+1}-2\right) \left( A_{p}-B_{p+1}\right) -\frac{3}{16}\right) .
\end{eqnarray*}%
Therefore, after some long but simple calculations using the
properties (\ref{Gamma_factorial}), (\ref%
{Gamma_reflection}) and \cite[Eqn. 1.2.3]{Lebedev} 
\begin{equation}
2^{2z-1}\Gamma \left( z\right) \Gamma \left( z+\frac{1}{2}\right) =\sqrt{\pi 
}\, \Gamma \left( 2z\right) ,  \label{Gamma_duplication}
\end{equation}%
of the gamma function, the asymptotic expansion of (\ref{lim_2F3})\ reads as%
\begin{eqnarray}
&&\frac{\cot \pi \nu }{2\nu\, \Gamma ^{2}\left( \nu +1\right) }\left( \frac{t}{%
2}\right) ^{2\nu }\,_{2}F_{3}\left( \left. 
\begin{array}{c}
\nu ,\nu +\frac{1}{2} \\ 
\nu +1,\nu +1,2\nu +1%
\end{array}%
\right\vert -t^{2}\right)  \label{lim_2F3_resultado} \\
&=&\frac{\cot \pi \nu }{2\nu }-\frac{\cot \pi \nu }{\pi t}+O\left( \frac{1}{%
t^{2}}\right) ,\qquad t\rightarrow \infty .  \notag
\end{eqnarray}%
Notice that to calculate the limit given in (\ref{lim_3F4}), we cannot
apply directly (\ref{pFp+1_infinito}) since we have a double pole ($%
a_{1}=a_{2}=1$). Nevertheless, we can still using (\ref{pFp+1_infinito}),
calculating the following asymptotic expansion: 
\begin{eqnarray*}
&&\frac{t^{2}}{4\pi \nu \left( 1-\nu ^{2}\right) }\,_{3}F_{4}\left( \left. 
\begin{array}{c}
1,1+\epsilon ,\frac{3}{2} \\ 
2,2,2-\nu ,2+\nu%
\end{array}%
\right\vert -t^{2}\right) \\
&=&-\frac{\Gamma \left( \epsilon -\frac{1}{2}\right) \cot \pi \nu }{2\pi
^{3/2}t\,\Gamma \left( 1+\epsilon \right) }+\frac{t^{-2+\epsilon }\cos
\left( 2t+\frac{\pi \epsilon }{2}\right) \csc \pi \nu }{2\pi \Gamma \left(
1+\epsilon \right) } \\
&&+\frac{1}{2\pi \nu \epsilon }+\frac{t^{-2\epsilon }\Gamma \left( \frac{1}{2%
}-\epsilon \right) \csc \pi \nu }{2\sqrt{\pi }\epsilon ^{2}\Gamma \left(
-\epsilon \right) \Gamma \left( 1-\nu -\epsilon \right) \Gamma \left( 1+\nu
-\epsilon \right) }+O\left( \frac{1}{t^{3}}\right) ,
\end{eqnarray*}%
and then calculating the limit $\epsilon \rightarrow 0$. For this purpose,
consider the following first order Taylor-series expansions as $\epsilon
\rightarrow 0$, 
\begin{eqnarray}
\Gamma \left( a-\epsilon \right) &\approx &\Gamma \left( a\right) \left[
1-\psi \left( a\right) \epsilon \right] ,  \label{Gamma_Taylor} \\
\frac{1}{\Gamma \left( a-\epsilon \right) } &\approx &\frac{1}{\Gamma \left(
a\right) }\left[ 1+\psi \left( a\right) \epsilon \right] ,
\label{Gamma_Taylor_2} \\
a^{\epsilon } &\approx &1+\log \left( a\right) \epsilon ,  \label{a^x_Taylor}
\end{eqnarray}%
where $\psi \left( z\right) =\Gamma ^{\prime }\left( z\right) /\Gamma \left(
z\right) $ denotes the digamma function \cite[Eqn. 5.2.2]{DLMF}. Also,
consider the following approximation (see \cite[Eqn. 5.7.1]{DLMF}),%
\begin{equation}
\Gamma \left( \epsilon \right) \approx \frac{1}{\epsilon }-\gamma ,\quad
\epsilon \rightarrow 0,  \label{Gamma_z->0}
\end{equation}%
where $\gamma =0.57721566\ldots $ denotes Euler's constant. Therefore,
taking into account (\ref{Gamma_Taylor})-(\ref{Gamma_z->0}), we have%
\begin{eqnarray}
&&\lim_{\epsilon \rightarrow 0}\, \frac{t^{2}}{4\pi \nu \left( 1-\nu
^{2}\right) }\,_{3}F_{4}\left( \left. 
\begin{array}{c}
1,1+\epsilon ,\frac{3}{2} \\ 
2,2,2-\nu ,2+\nu%
\end{array}%
\right\vert -t^{2}\right)  \label{lim_3F4_resultado} \\
&\sim &\frac{1}{2\pi \nu }\left[ \log \left( \frac{t^{2}}{4}\right) -\psi
\left( 1+\nu \right) -\psi \left( 1-\nu \right) \right] ,\quad t\rightarrow
\infty ,  \notag
\end{eqnarray}%
where we have considered that \cite[Eqn. 1.3.8]{Lebedev} 
\begin{equation*}
\psi \left( \frac{1}{2}\right) =-\gamma -2\log 2.
\end{equation*}%
Taking into account (\ref{lim_2F3_resultado})\ and (\ref%
{lim_3F4_resultado}), and applying the following properties of the digamma
function \cite[Eqns. 1.3.3\&4]{Lebedev} 
\begin{eqnarray}
\psi \left( z+1\right) &=&\frac{1}{z}+\psi \left( z\right) ,
\label{Digamma_1} \\
\psi \left( 1-z\right) -\psi \left( z\right) &=&\pi \cot \pi z,
\label{Digamma_2}
\end{eqnarray}%
we arrive at%
\begin{eqnarray}
&&\lim_{t\rightarrow \infty }\, \frac{1}{\pi \nu }\left[ \frac{\pi \cot \pi \nu
\ \left( t/2\right) ^{2\nu }}{2\, \Gamma ^{2}\left( \nu +1\right) }%
\,_{2}F_{3}\left( \left. 
\begin{array}{c}
\nu ,\nu +\frac{1}{2} \\ 
\nu +1,\nu +1,2\nu +1%
\end{array}%
\right\vert -t^{2}\right) \right.  \label{lim_Int_t->inf} \\
&&-\left. \log t+\frac{t^{2}}{4\left( 1-\nu ^{2}\right) }\,_{3}F_{4}\left(
\left. 
\begin{array}{c}
1,1,\frac{3}{2} \\ 
2,2,2-\nu ,2+\nu%
\end{array}%
\right\vert -t^{2}\right) \right]  \notag \\
&=&-\frac{1}{\pi \nu }\left[ \frac{1}{2\nu }+\psi \left( \nu \right) +\log 2%
\right] .  \notag
\end{eqnarray}%
Finally, according to (\ref{Int_Jnu_Ynu_indefinida}) and (\ref%
{lim_Int_t->inf}), we conclude (\ref{Int_Jnu_Ynu}).
\end{proof}

Next, we will calculate the integrals given in the integral representation of 
$\partial Y_{\nu }/\partial \nu $ given in (\ref{DYnu_int_Dunster}).

\begin{theorem}
If $z\neq 0$, $\left\vert \mathrm{arg}\ z\right\vert <\pi $ and $\nu >0$, $%
\nu \notin 
\mathbb{Z}
$, the following integral holds true:%
\begin{eqnarray}
&&\int_{z}^{\infty }\frac{Y_{\nu }^{2}\left( t\right) }{t}dt  \label{Int_Y2}
\\
&=&\frac{1}{2\pi ^{2}\nu }\left[ \left( \frac{z}{2}\right) ^{-2\nu }\Gamma
^{2}\left( \nu \right) \,_{2}F_{3}\left( \left. 
\begin{array}{c}
-\nu ,\frac{1}{2}-\nu \\ 
1-\nu ,1-\nu ,1-2\nu%
\end{array}%
\right\vert -z^{2}\right) \right.  \notag \\
&&-\left. \left( \frac{z}{2}\right) ^{2\nu }\Gamma ^{2}\left( -\nu \right)
\cos ^{2}\pi \nu \,_{2}F_{3}\left( \left. 
\begin{array}{c}
\nu ,\nu +\frac{1}{2} \\ 
\nu +1,\nu +1,2\nu +1%
\end{array}%
\right\vert -z^{2}\right) \right]  \notag \\
&&-\frac{1+2\cot ^{2}\pi \nu }{2\nu }-\frac{2\cot \pi \nu }{\pi \nu }\left[ 
\frac{z^{2}}{4\left( 1-\nu ^{2}\right) }\,_{3}F_{4}\left( \left. 
\begin{array}{c}
1,1,\frac{3}{2} \\ 
2,2,2-\nu ,2+\nu%
\end{array}%
\right\vert -z^{2}\right) \right.  \notag \\
&&\quad \left. 
\begin{array}{c}
\displaystyle
\\ 
\displaystyle%
\end{array}%
+\log \left( \frac{2}{z}\right) +\frac{1}{2\nu }+\psi \left( \nu \right) %
\right] .  \notag
\end{eqnarray}
\end{theorem}

\begin{proof}
First, let us calculate the indefinite integral of (\ref{Int_Y2}). By using
the definition given in (\ref{Ynu_def}) of the Bessel function of the second kind $%
Y_{\nu }\left( z\right) $, we have that 
\begin{eqnarray}
&&\int \frac{Y_{\nu }^{2}\left( t\right) }{t}dt  \label{Int_Y2_indefinida} \\
&=&\cot ^{2}\pi \nu \int \frac{J_{\nu }^{2}\left( t\right) }{t}dt+\csc
^{2}\pi \nu \int \frac{J_{-\nu }^{2}\left( t\right) }{t}dt  \notag \\
&&-2\frac{\cos \pi \nu }{\sin ^{2}\pi \nu }\int \frac{J_{\nu }\left(
t\right) J_{-\nu }\left( t\right) }{t}dt.  \notag
\end{eqnarray}%
Notice that we have already calculated the first integral given on the RHS\ of (\ref{Int_Y2_indefinida}%
)\ in (\ref{Int_Jnu^2b}), thus the second integral on
the RHS of (\ref{Int_Y2_indefinida})\ is precisely (\ref{Int_Jnu^2b})\ performing the change $%
\nu \rightarrow -\nu $. Also, we have already calculated the third integral on the RHS\ of (\ref%
{Int_Y2_indefinida})\ in (\ref{Int_Jnu_J-nu_resultado}).
Collecting all these results, we have%
\begin{eqnarray}
&&\int \frac{Y_{\nu }^{2}\left( t\right) }{t}dt
\label{Int_Y2_indefinida_resultado} \\
&=&\frac{\cot ^{2}\pi \nu \ \left( t/2\right) ^{2\nu }}{2\nu\, \Gamma
^{2}\left( \nu +1\right) }\,_{2}F_{3}\left( \left. 
\begin{array}{c}
\nu ,\nu +\frac{1}{2} \\ 
\nu +1,\nu +1,2\nu +1%
\end{array}%
\right\vert -t^{2}\right)   \notag \\
&&-\frac{\cot ^{2}\pi \nu \ \left( t/2\right) ^{-2\nu }}{2\nu\, \Gamma
^{2}\left( 1-\nu \right) }\,_{2}F_{3}\left( \left. 
\begin{array}{c}
-\nu ,\frac{1}{2}-\nu  \\ 
1-\nu ,1-\nu ,1-2\nu 
\end{array}%
\right\vert -t^{2}\right)   \notag \\
&&-\frac{2\cot \pi \nu }{\pi \nu }\left[ \log t-\frac{t^{2}}{4\nu \left(
1-\nu ^{2}\right) }\,_{3}F_{4}\left( \left. 
\begin{array}{c}
1,1,\frac{3}{2} \\ 
2,2,2-\nu ,2+\nu 
\end{array}%
\right\vert -t^{2}\right) \right] .  \notag
\end{eqnarray}%
To calculate (\ref{Int_Y2_indefinida}) with the integration limits
given in (\ref{Int_Y2}), we have to consider the asymptotic expansion (\ref%
{lim_2F3_resultado}), replacing $\nu \rightarrow \pm \nu $%
\begin{eqnarray}
&&\frac{\pm \left( t/2\right) ^{\pm 2\nu }}{2\nu\, \Gamma ^{2}\left( 1\pm \nu
\right) }\,_{2}F_{3}\left( \left. 
\begin{array}{c}
\pm \nu ,\frac{1}{2}\pm \nu  \\ 
1\pm \nu ,1\pm \nu ,1\pm 2\nu 
\end{array}%
\right\vert -t^{2}\right)   \label{Asint_Y2_1} \\
&=&\frac{\pm 1}{2\nu }-\frac{1}{\pi t}+O\left( \frac{1}{t^{2}}\right)
,\qquad t\rightarrow \infty .  \notag
\end{eqnarray}%
Also, considering the asymptotic expansion (\ref{lim_3F4_resultado}) and taking
into account the properties of the digamma function given in (\ref{Digamma_1}) and (%
\ref{Digamma_2}), we have%
\begin{eqnarray}
&&\frac{t^{2}}{4\pi \nu \left( 1-\nu ^{2}\right) }\,_{3}F_{4}\left( \left. 
\begin{array}{c}
1,1,\frac{3}{2} \\ 
2,2,2-\nu ,2+\nu 
\end{array}%
\right\vert -t^{2}\right)   \label{Asint_Y2_2} \\
&\sim &\frac{1}{2\pi \nu }\left[ \log \left( \frac{t^{2}}{4}\right) -\frac{1%
}{\nu }-2\psi \left( \nu \right) -\pi \cot \pi \nu \right] ,\quad
t\rightarrow \infty .  \notag
\end{eqnarray}%
Therefore, taking into account the indefinite integral (\ref%
{Int_Y2_indefinida_resultado}) and the asymptotic expansions (\ref%
{Asint_Y2_1}) and (\ref{Asint_Y2_2}), after some simple calculations wherein
we apply the reflection formula of the gamma function (\ref%
{Gamma_reflection}), we arrive at (\ref{Int_Y2}).
\end{proof}

Finally, according to the integral representation given in (\ref%
{DJnu_int_Dunster}), and the integrals calculated in (\ref{Int_Jnu^2b}) and (%
\ref{Int_Jnu_Ynu}), we express in closed-form the order derivative of
the Bessel function as, 
\begin{eqnarray}
&&\frac{\partial J_{\nu }\left( z\right) }{\partial \nu }
\label{DJnu_closed_form} \\
&=&\frac{-\pi J_{-\nu }\left( z\right) \csc \pi \nu }{2\Gamma ^{2}\left( \nu
+1\right) }\left( \frac{z}{2}\right) ^{2\nu }\,_{2}F_{3}\left( \left. 
\begin{array}{c}
\nu ,\nu +\frac{1}{2} \\ 
\nu +1,\nu +1,2\nu +1%
\end{array}%
\right\vert -z^{2}\right)  \notag \\
&&-J_{\nu }\left( z\right) \left[ \frac{z^{2}}{4\left( 1-\nu ^{2}\right) }%
\,_{3}F_{4}\left( \left. 
\begin{array}{c}
1,1,\frac{3}{2} \\ 
2,2,2-\nu ,2+\nu%
\end{array}%
\right\vert -z^{2}\right) \right.  \notag \\
&&\quad \left. 
\begin{array}{c}
\displaystyle
\\ 
\displaystyle%
\end{array}%
+\log \left( \frac{2}{z}\right) +\frac{1}{2\nu }+\psi \left( \nu \right) %
\right] ,  \notag
\end{eqnarray}%
where we have taken into account the definition of $Y_{\nu }\left( z\right) $ given in
(\ref{Ynu_def}). Note that (\ref{DJnu_closed_form})\ is equivalent to the
result obtained by Brychov in (\ref{D_J_nu_Brychov}).

Similarly, substituting (\ref{Int_Y2}) and (\ref{Int_Jnu_Ynu}) in (\ref%
{DYnu_int_Dunster}), after some simplification, we arrive at,%
\begin{eqnarray}
&&\frac{\partial Y_{\nu }\left( z\right) }{\partial \nu }
\label{DYnu_closed_form} \\
&=&J_{\nu }\left( z\right) \left[ \frac{\Gamma ^{2}\left( \nu \right) }{2\pi 
}\left( \frac{z}{2}\right) ^{-2\nu }\,_{2}F_{3}\left( \left. 
\begin{array}{c}
-\nu ,\frac{1}{2}-\nu \\ 
1-\nu ,1-\nu ,1-2\nu%
\end{array}%
\right\vert -z^{2}\right) -\pi \csc ^{2}\pi \nu \right]  \notag \\
&&-\frac{\cos \pi \nu }{2\pi }\Gamma ^{2}\left( -\nu \right) J_{-\nu }\left(
z\right) \left( \frac{z}{2}\right) ^{2\nu }\,_{2}F_{3}\left( \left. 
\begin{array}{c}
\nu ,\nu +\frac{1}{2} \\ 
\nu +1,\nu +1,2\nu +1%
\end{array}%
\right\vert -z^{2}\right)  \notag \\
&&+\left[ \log \left( \frac{2}{z}\right) +\frac{1}{2\nu }+\psi \left( \nu
\right) +\frac{z^{2}}{4\left( 1-\nu ^{2}\right) }\,_{3}F_{4}\left( \left. 
\begin{array}{c}
1,1,\frac{3}{2} \\ 
2,2,2-\nu ,2+\nu%
\end{array}%
\right\vert -z^{2}\right) \right]  \notag \\
&&\times \left( Y_{\nu }\left( z\right) -2\cot \pi \nu \ J_{\nu }\left(
z\right) \right) ,  \notag
\end{eqnarray}%
which is equivalent to the result given in \cite{BrychovNew}.

Finally, according to (\ref{Int_Jnu^2b})\ and (\ref{Int_Y2}), we rewrite (%
\ref{D(Jnu*Ynu)_int}) in closed-form\ as,%
\begin{eqnarray}
&&\frac{\partial }{\partial \nu }\left( J_{\nu }\left( z\right) Y_{\nu
}\left( z\right) \right)  \label{D(Jnu*Ynu)_closed_form} \\
&=&\frac{J_{-\nu }\left( z\right) }{2\pi }\left( \frac{z}{2}\right) ^{2\nu
}\Gamma ^{2}\left( -\nu \right)  \notag \\
&&\times \left[ J_{-\nu }\left( z\right) -2\cos \pi \nu \,J_{\nu }\left(
z\right) \right] \,_{2}F_{3}\left( \left. 
\begin{array}{c}
\nu ,\nu +\frac{1}{2} \\ 
\nu +1,\nu +1,2\nu +1%
\end{array}%
\right\vert -z^{2}\right)  \notag \\
&&+J_{\nu }^{2}\left( z\right) \left\{ \frac{\left( z/2\right) ^{-2\nu }}{%
2\pi }\Gamma ^{2}\left( \nu \right) \,_{2}F_{3}\left( \left. 
\begin{array}{c}
-\nu ,+\frac{1}{2}-\nu \\ 
1-\nu ,1-\nu ,1-2\nu%
\end{array}%
\right\vert -z^{2}\right) \right.  \notag \\
&&-\pi \csc ^{2}\pi \nu -2\cot \pi \nu  \notag \\
&&\times \left. \left[ \frac{z^{2}}{4\left( 1-\nu ^{2}\right) }%
\,_{3}F_{4}\left( \left. 
\begin{array}{c}
1,1,\frac{3}{2} \\ 
2,2,2-\nu ,2+\nu%
\end{array}%
\right\vert -z^{2}\right) +\log \left( \frac{2}{z}\right) +\frac{1}{2\nu }%
+\psi \left( \nu \right) \right] \right\} .  \notag
\end{eqnarray}

\section{Order derivatives for modified Bessel functions\label{Section:
Modified Bessel functions}}

Similar integrals as in the previous Section can be calculated replacing
Bessel functions by modified Bessel functions. Here we collect the results
with a sketch of the proof.

\begin{theorem}
If $\nu >0$, the following integral holds true:%
\begin{equation}
\int_{0}^{z}\frac{I_{\nu }^{2}\left( t\right) }{t}dt=\frac{\left( z/2\right)
^{2\nu }}{2\nu\, \Gamma ^{2}\left( \nu +1\right) }\,_{2}F_{3}\left( \left. 
\begin{array}{c}
\nu ,\nu +\frac{1}{2} \\ 
\nu +1,\nu +1,2\nu +1%
\end{array}%
\right\vert z^{2}\right) .  \label{Int_I2}
\end{equation}
\end{theorem}

\begin{proof}
Integrate term by term the following power series (Cauchy product) \cite[%
Eqn. 10.31.3]{DLMF}, 
\begin{equation}
I_{\nu }\left( z\right) I_{\mu }\left( z\right) =\left( \frac{z}{2}\right)
^{\mu +\nu }\sum_{n=0}^{\infty }\frac{\left( \nu +\mu +n+1\right) _{n}\left(
z/2\right) ^{2n}}{n!\, \Gamma \left( n+\mu +1\right) \Gamma \left( n+\nu
+1\right) },  \label{I_nu*I_mu}
\end{equation}%
taking $\mu =\nu $, and recasting the result as a hypergeometric series.
\end{proof}

\begin{remark}
If we take $\mu =-\nu $ in (\ref{I_nu*I_mu}), we will arrive at%
\begin{eqnarray}
&&\int \frac{I_{-\nu }\left( t\right) I_{\nu }\left( t\right) }{t}dt
\label{Int_Inu_I-nu_indefinida} \\
&=&\frac{\sin \pi \nu }{\pi \nu }\left[ \log t+\frac{t^{2}}{4\left( 1-\nu
^{2}\right) }\,_{3}F_{4}\left( \left. 
\begin{array}{c}
1,1,\frac{3}{2} \\ 
2,2,2+\nu ,2-\nu%
\end{array}%
\right\vert t^{2}\right) \right] .  \notag
\end{eqnarray}
\end{remark}

\begin{theorem}
If $z\neq 0$, $\left\vert \mathrm{arg}\ z\right\vert <\pi $, and $\nu >0$, $%
\nu \notin 
\mathbb{Z}
$, the following integral holds true: 
\begin{eqnarray}
&&\int_{z}^{\infty }\frac{I_{\nu }\left( t\right) K_{\nu }\left( t\right) }{t%
}dt  \label{Int_Inu_Knu} \\
&=&\frac{1}{2\nu }\left[ \frac{\pi \csc \pi \nu \ \left( z/2\right) ^{2\nu }%
}{2\, \Gamma ^{2}\left( \nu +1\right) }\,_{2}F_{3}\left( \left. 
\begin{array}{c}
\nu ,\nu +\frac{1}{2} \\ 
\nu +1,\nu +1,2\nu +1%
\end{array}%
\right\vert z^{2}\right) \right.  \notag \\
&&-\left. \frac{z^{2}}{4\left( 1-\nu ^{2}\right) }\,_{3}F_{4}\left( \left. 
\begin{array}{c}
1,1,\frac{3}{2} \\ 
2,2,2-\nu ,2+\nu%
\end{array}%
\right\vert z^{2}\right) +\log \left( \frac{2}{z}\right) +\psi \left( \nu
\right) +\frac{1}{2\nu }\right] .  \notag
\end{eqnarray}
\end{theorem}

\begin{proof}
Expanding $K_{\nu }\left( z\right) $ in (\ref{Int_Inu_Knu})\ and then using (%
\ref{Int_I2})\ and (\ref{Int_Inu_I-nu_indefinida}), we obtain the following
result for the indefinite integral:%
\begin{eqnarray}
&&\int \frac{I_{\nu }\left( t\right) K_{\nu }\left( t\right) }{t}dt
\label{Int_Inu_Knu_indefinida} \\
&=&\frac{1}{2\nu }\left[ \log t-\frac{\pi \csc \pi \nu \ \left( t/2\right)
^{2\nu }}{2\, \Gamma ^{2}\left( \nu +1\right) }\,_{2}F_{3}\left( \left. 
\begin{array}{c}
\nu ,\nu +\frac{1}{2} \\ 
\nu +1,\nu +1,2\nu +1%
\end{array}%
\right\vert t^{2}\right) \right.  \notag \\
&&+\left. \frac{t^{2}}{4\left( 1-\nu ^{2}\right) }\,_{3}F_{4}\left( \left. 
\begin{array}{c}
1,1,\frac{3}{2} \\ 
2,2,2-\nu ,2+\nu%
\end{array}%
\right\vert t^{2}\right) \right] .  \notag
\end{eqnarray}%
To obtain (\ref{Int_Inu_Knu}), perform the asymptotic calculation
of the hypergeometric functions given in (\ref{Int_Inu_Knu_indefinida}),
rewriting (\ref{pFp+1_infinito})\ as%
\begin{eqnarray}
&&_{p}F_{p+1}\left( \left. 
\begin{array}{c}
a_{1},\ldots ,a_{p} \\ 
b_{1},\ldots b_{p+1}%
\end{array}%
\right\vert z\right)  \label{pFp+1_Exp} \\
&\sim &\frac{\prod_{j=1}^{p+1}\Gamma \left( b_{j}\right) }{2\sqrt{\pi }%
\prod_{k=1}^{p}\Gamma \left( a_{k}\right) }z^{\chi }e^{2\sqrt{z}}\left[
1+O\left( \frac{1}{\sqrt{z}}\right) \right]  \notag \\
&&+\frac{\prod_{j=1}^{p+1}\Gamma \left( b_{j}\right) }{\prod_{k=1}^{p}\Gamma
\left( a_{k}\right) }\sum_{k=1}^{p}\frac{\Gamma \left( a_{k}\right)
\prod_{j=1,j\neq k}^{p}\Gamma \left( a_{j}-a_{k}\right) }{%
\prod_{j=1}^{p+1}\Gamma \left( b_{j}-a_{k}\right) }\left( -z\right) ^{-a_{k}}%
\left[ 1+O\left( \frac{1}{z}\right) \right] .  \notag
\end{eqnarray}
\end{proof}

\begin{theorem}
If $z\neq 0$, $\left\vert \mathrm{arg}\ z\right\vert \leq \pi $, and $\nu
\notin 
\mathbb{Z}
$, $\nu \neq \pm 1/2,\pm 3/2$, the following integral holds true:%
\begin{eqnarray}
&&\int_{z}^{\infty }\frac{K_{\nu }^{2}\left( t\right) }{t}dt
\label{Int_Knu^2} \\
&=&\frac{1}{8\nu }\left\{ \left( \frac{z}{2}\right) ^{-2\nu }\Gamma
^{2}\left( \nu \right) \,_{2}F_{3}\left( \left. 
\begin{array}{c}
\nu ,\frac{1}{2}+\nu \\ 
1+\nu ,1+\nu ,1+2\nu%
\end{array}%
\right\vert z^{2}\right) \right.  \notag \\
&&-\left( \frac{z}{2}\right) ^{2\nu }\Gamma ^{2}\left( -\nu \right)
\,_{2}F_{3}\left( \left. 
\begin{array}{c}
-\nu ,\frac{1}{2}-\nu \\ 
1-\nu ,1-\nu ,1-2\nu%
\end{array}%
\right\vert z^{2}\right)  \notag \\
&&+4\pi \csc \pi \nu \left[ \log \left( \frac{z}{2}\right) +\frac{z^{2}}{%
4\left( 1-\nu ^{2}\right) }\,_{3}F_{4}\left( \left. 
\begin{array}{c}
1,1,\frac{3}{2} \\ 
2,2,2-\nu ,2+\nu%
\end{array}%
\right\vert z^{2}\right) \right.  \notag \\
&&\quad \left. \left. 
\begin{array}{c}
\displaystyle
\\ 
\displaystyle%
\end{array}%
-\frac{1}{2\nu }-\psi \left( \nu \right) -\frac{\pi }{2}\cot \pi \nu \right]
\right\} .  \notag
\end{eqnarray}
\end{theorem}

\begin{proof}
Consider the definition given in (\ref{Knu_def}) for $K_{\nu }\left( z\right) $\ to write the following. 
\begin{eqnarray*}
&&\int \frac{K_{\nu }^{2}\left( t\right) }{t}dt \\
&=&\frac{\pi ^{2}}{4}\csc ^{2}\pi \nu \left[ \int \frac{I_{-\nu }^{2}\left(
t\right) }{t}dt+\int \frac{I_{\nu }^{2}\left( t\right) }{t}dt-2\int \frac{%
I_{-\nu }\left( t\right) I_{\nu }\left( t\right) }{t}dt\right] .
\end{eqnarray*}%
Taking into account the results given in (\ref{Int_I2})\ and (\ref%
{Int_Inu_I-nu_indefinida}), we obtain%
\begin{eqnarray}
&&\int \frac{K_{\nu }^{2}\left( t\right) }{t}dt  \label{Int_K2_indefinida} \\
&=&\frac{\pi ^{2}}{4}\csc ^{2}\pi \nu \left\{ \frac{\left( t/2\right) ^{2\nu
}}{2\nu\, \Gamma ^{2}\left( \nu +1\right) }\,_{2}F_{3}\left( \left. 
\begin{array}{c}
\nu ,\nu +\frac{1}{2} \\ 
\nu +1,\nu +1,2\nu +1%
\end{array}%
\right\vert t^{2}\right) \right.   \notag \\
&&-\frac{\left( t/2\right) ^{-2\nu }}{2\nu\, \Gamma ^{2}\left( 1-\nu \right) }%
\,_{2}F_{3}\left( \left. 
\begin{array}{c}
-\nu ,\frac{1}{2}-\nu  \\ 
1-\nu ,1-\nu ,1-2\nu 
\end{array}%
\right\vert t^{2}\right)   \notag \\
&&-\left. 2\frac{\sin \pi \nu }{\pi \nu }\left[ \log t+\frac{t^{2}}{4\left(
1-\nu ^{2}\right) }\,_{3}F_{4}\left( \left. 
\begin{array}{c}
1,1,\frac{3}{2} \\ 
2,2,2+\nu ,2-\nu 
\end{array}%
\right\vert t^{2}\right) \right] \right\} .  \notag
\end{eqnarray}%
According to (\ref{pFp+1_Exp}), we have the following asymptotic
expansion as $t\rightarrow \infty $ 
\begin{eqnarray}
&&\pm \frac{\left( t/2\right) ^{\pm 2\nu }}{2\nu\, \Gamma ^{2}\left( \nu \pm
1\right) }\,_{2}F_{3}\left( \left. 
\begin{array}{c}
\pm \nu ,\frac{1}{2}\pm \nu  \\ 
1\pm \nu ,1\pm \nu ,1\pm 2\nu 
\end{array}%
\right\vert t^{2}\right)   \label{2F3_inf} \\
&\approx &\frac{e^{2t}}{4\pi t^{2}}+\frac{i\left( -1\right) ^{\mp \nu }}{\pi
t}\pm \frac{\left( -1\right) ^{\mp \nu }}{2\nu }.  \notag
\end{eqnarray}%
Also, 
\begin{eqnarray}
&&\lim_{\epsilon \rightarrow 0}\, \frac{t^{2}}{4\left( 1-\nu ^{2}\right) }%
\,_{3}F_{4}\left( \left. 
\begin{array}{c}
1,1+\epsilon ,\frac{3}{2} \\ 
2,2,2+\nu ,2-\nu 
\end{array}%
\right\vert t^{2}\right)   \label{3F4_inf} \\
&\approx &\frac{i\nu \cot \pi \nu }{t}+\frac{\nu e^{2t}\csc \pi \nu }{4t^{2}}%
+\frac{\psi \left( 1+\nu \right) +\psi \left( 1-\nu \right) -\log \left(
-t^{2}\right) }{2}+\log 2.  \notag
\end{eqnarray}%
Taking into account (\ref{2F3_inf})\ and (\ref{3F4_inf})\ in (\ref%
{Int_K2_indefinida}), after some simplification, we eventually arrive at (%
\ref{Int_Knu^2}).
\end{proof}

Next, following a similar derivation as the one given in \cite{Dunster} for
the integral representation of $\partial J_{\nu }/\partial \nu $, we obtain
an integral representation of $\partial I_{\nu }/\partial \nu $.

\begin{theorem}
For $\nu >0$ and $z\neq 0$, $\left\vert \mathrm{arg}\ z\right\vert \leq \pi $%
, we have 
\begin{equation}
\frac{\partial I_{\nu }\left( z\right) }{\partial \nu }=-2\nu \left[ I_{\nu
}\left( z\right) \int_{z}^{\infty }\frac{K_{\nu }\left( t\right) I_{\nu
}\left( t\right) }{t}dt+K_{\nu }\left( z\right) \int_{0}^{z}\frac{I_{\nu
}^{2}\left( t\right) }{t}dt\right] .  \label{DInu_int}
\end{equation}
\end{theorem}

\begin{proof}
Any linear combination of the modified Bessel functions $I_{\nu }\left(
z\right) $ and $K_{\nu }\left( z\right) $ satisfies the following ODE \cite[Eqn. 5.7.7]{Lebedev}, 
\begin{equation}
u^{\prime \prime }\left( z\right) +\frac{1}{z}u^{\prime }\left( z\right)
-\left( 1+\frac{\nu ^{2}}{z^{2}}\right) u\left( z\right) =0.
\label{Eqn_Modified_Bessel}
\end{equation}%
Consider now $u\left( z\right) =I_{\nu }\left( z\right) $, and perform\ the
derivative with respect to the order in (\ref{Eqn_Modified_Bessel}), to
obtain%
\begin{equation*}
\frac{d^{2}}{dz^{2}}\left( \frac{\partial I_{\nu }\left( z\right) }{\partial
\nu }\right) +\frac{1}{z}\frac{d\,}{dz}\left( \frac{\partial I_{\nu }\left(
z\right) }{\partial \nu }\right) -\left( 1+\frac{\nu ^{2}}{z^{2}}\right) 
\frac{\partial I_{\nu }\left( z\right) }{\partial \nu }=\frac{2\nu }{z^{2}}%
I_{\nu }\left( z\right) .
\end{equation*}%
Applying now the method of variation of parameters \cite[Sect. 16.516]%
{Gradshteyn}, taking into account the following Wronskian \cite[Eqn. 5.9.5]%
{Lebedev}%
\begin{equation*}
W\left[ I_{\nu }\left( z\right) ,K_{\nu }\left( z\right) \right] =-\frac{1}{z%
},
\end{equation*}%
the general solution of (\ref{Eqn_Modified_Bessel})\ is given by%
\begin{eqnarray}
\frac{\partial I_{\nu }\left( z\right) }{\partial \nu } &=&-2\nu \left[
I_{\nu }\left( z\right) \int_{z}^{\infty }\frac{K_{\nu }\left( t\right)
I_{\nu }\left( t\right) }{t}dt+K_{\nu }\left( z\right) \int_{0}^{z}\frac{%
I_{\nu }^{2}\left( t\right) }{t}dt\right]  \label{DInu_int_1} \\
&&+a_{\nu }I_{\nu }\left( z\right) +b_{\nu }K_{\nu }\left( z\right) ,  \notag
\end{eqnarray}%
where $a_{\nu }$ and $b_{\nu }$ are constants that can be determined as
follows. First, notice that from the series representation (\ref{DInu_series}%
), for $\nu >0$ we have that%
\begin{equation}
\lim_{z\rightarrow 0}\, \frac{\partial I_{\nu }\left( z\right) }{\partial \nu }%
=\lim_{z\rightarrow 0}\, I_{\nu }\left( z\right) \log \left( \frac{z}{2}\right)
=0,  \label{Lim_DInu_z->0}
\end{equation}%
since, according to \cite[Eqn. 5.16.4]{Lebedev},%
\begin{equation}
I_{\nu }\left( z\right) \approx \frac{\left( z/2\right) ^{\nu }}{\Gamma
\left( 1+\nu \right) },\quad z\rightarrow 0.  \label{Lim_Inu_z->0}
\end{equation}%
Also, from (\ref{Int_I2}), we have%
\begin{equation}
\int_{0}^{z}\frac{I_{\nu }^{2}\left( t\right) }{t}dt\approx \frac{\left(
z/2\right) ^{2\nu }}{2\nu \Gamma ^{2}\left( \nu +1\right) },\quad
z\rightarrow 0,  \label{Lim_Int_I2_z->0}
\end{equation}%
and from (\ref{Int_Inu_Knu}), we have as well%
\begin{equation}
\int_{z}^{\infty }\frac{I_{\nu }\left( t\right) K_{\nu }\left( t\right) }{t}%
dt\approx \frac{1}{2\nu }\log \left( \frac{2}{z}\right) ,\quad z\rightarrow
0.  \label{Lim_Int_Inu_Knu_z->0}
\end{equation}%
Therefore, performing the limit $z\rightarrow 0$ on both sides of (\ref%
{DInu_int_1})\ and taking into account (\ref{Lim_DInu_z->0})-(\ref%
{Lim_Int_Inu_Knu_z->0}), we conclude that $b_{\nu }=0$ since $K_{\nu
}\left( z\right) $ is divergent as $z\rightarrow 0$ \cite[Eqn. 5.16.4]%
{Lebedev}. Thereby, rewrite (\ref{DInu_int_1})\ as%
\begin{eqnarray}
&&\frac{\partial I_{\nu }\left( z\right) }{\partial \nu }  \label{DInu_int_2}
\\
&=&-2\nu \left\{ I_{\nu }\left( z\right) \left[ a_{\nu }+\int_{z}^{\infty }%
\frac{K_{\nu }\left( t\right) I_{\nu }\left( t\right) }{t}dt\right] +K_{\nu
}\left( z\right) \int_{0}^{z}\frac{I_{\nu }^{2}\left( t\right) }{t}%
dt\right\} .  \notag
\end{eqnarray}%
Consider now the following asymptotic expansions \cite[Eqns. 10.40.1-2]%
{DLMF} as $z\rightarrow \infty $,%
\begin{eqnarray}
I_{\nu }\left( z\right) &=&\frac{e^{z}}{\sqrt{2\pi z}}\left[ 1-\frac{4\nu
^{2}-1}{8z}+O\left( \frac{1}{z^{2}}\right) \right] ,  \label{Inu_z->inf} \\
K_{\nu }\left( z\right) &=&\sqrt{\frac{\pi }{2z}}e^{-z}\left[ 1+\frac{4\nu
^{2}-1}{8z}+O\left( \frac{1}{z^{2}}\right) \right] .  \label{Knu_z->inf}
\end{eqnarray}%
On the one hand, performing the order derivative in (\ref{Inu_z->inf}), the
asymptotic expansion on the LHS\ of (\ref{Inu_z->inf})\ is 
\begin{equation}
\frac{\partial I_{\nu }\left( z\right) }{\partial \nu }\approx -\frac{\nu
\,e^{z}}{\sqrt{2\pi }z^{3/2}},\quad z\rightarrow \infty .
\label{DInu_z->inf}
\end{equation}%
On the other hand, taking into account (\ref{Inu_z->inf})\ and (\ref%
{Knu_z->inf}), we have%
\begin{equation}
\int_{z}^{\infty }\frac{K_{\nu }\left( t\right) I_{\nu }\left( t\right) }{t}%
dt\approx \frac{1}{2z},\quad z\rightarrow \infty .
\label{Int_Inu_Knu_z->inf}
\end{equation}%
Also, from (\ref{Int_I2})\ and (\ref{pFp+1_Exp}), we have, 
\begin{equation}
\int_{0}^{z}\frac{I_{\nu }^{2}\left( t\right) }{t}dt\approx \frac{e^{2z}}{%
4\pi z^{2}},\quad z\rightarrow \infty .  \label{Int_Inu2_z->inf}
\end{equation}%
Therefore, from (\ref{Inu_z->inf}), (\ref{Knu_z->inf}), (\ref%
{Int_Inu_Knu_z->inf}), and (\ref{Int_Inu2_z->inf}), the asymptotic expansion
on the RHS\ of (\ref{DInu_int_2})\ is%
\begin{equation}
\frac{\partial I_{\nu }\left( z\right) }{\partial \nu }\approx -2\nu \frac{%
e^{z}}{\sqrt{2\pi z}}\left( \frac{1}{2z}+a_{\nu }\right) ,\quad z\rightarrow
\infty .  \label{DInu_z->inf_2}
\end{equation}%
Comparing (\ref{DInu_z->inf})\ to (\ref{DInu_z->inf_2}), we conclude that $%
a_{\nu }=0$, hence we obtain the integral representation given in (\ref%
{DInu_int}).
\end{proof}

Once we have set the integral representation of $\partial I_{\nu }/\partial
\nu $, applying the results given in (\ref{Int_I2})\ and (\ref{Int_Inu_Knu}%
), we can rewrite (\ref{DInu_int})\ in closed-form as follows: 
\begin{eqnarray}
&&\frac{\partial I_{\nu }\left( z\right) }{\partial \nu }
\label{DInu_resultado} \\
&=&I_{\nu }\left( z\right) \left[ \frac{z^{2}}{4\left( 1-\nu ^{2}\right) }%
\,_{3}F_{4}\left( \left. 
\begin{array}{c}
1,1,\frac{3}{2} \\ 
2,2,2-\nu ,2+\nu%
\end{array}%
\right\vert z^{2}\right) +\log \left( \frac{z}{2}\right) -\psi \left( \nu
\right) -\frac{1}{2\nu }\right]  \notag \\
&&-I_{-\nu }\left( z\right) \frac{\pi \csc \pi \nu }{2\, \Gamma ^{2}\left( \nu
+1\right) }\left( \frac{z}{2}\right) ^{2\nu }\,_{2}F_{3}\left( \left. 
\begin{array}{c}
\nu ,\frac{1}{2}+\nu \\ 
1+\nu ,1+\nu ,1+2\nu%
\end{array}%
\right\vert z^{2}\right) ,  \notag
\end{eqnarray}%
which is equivalent to the result given in \cite{BrychovNew}.

Also, according to (\ref{DKnu_(DInu)}) and the above result (\ref%
{DInu_resultado}), after some simplification, we arrive at%
\begin{eqnarray}
&&\frac{\partial K_{\nu }\left( z\right) }{\partial \nu }
\label{DKnu_resultado} \\
&=&\frac{\pi }{2}\csc \pi \nu \left\{ \pi \cot \pi \nu \,I_{\nu }\left(
z\right) -\left[ I_{\nu }\left( z\right) +I_{-\nu }\left( z\right) \right] 
\begin{array}{c}
\displaystyle
\\ 
\displaystyle%
\end{array}%
\right.  \notag \\
&&\left. \left[ \frac{z^{2}}{4\left( 1-\nu ^{2}\right) }\,_{3}F_{4}\left(
\left. 
\begin{array}{c}
1,1,\frac{3}{2} \\ 
2,2,2-\nu ,2+\nu%
\end{array}%
\right\vert z^{2}\right) +\log \left( \frac{z}{2}\right) -\psi \left( \nu
\right) -\frac{1}{2\nu }\right] \right\}  \notag \\
&&+\frac{1}{4}\left\{ I_{-\nu }\left( z\right) \Gamma ^{2}\left( -\nu
\right) \left( \frac{z}{2}\right) ^{2\nu }\,_{2}F_{3}\left( \left. 
\begin{array}{c}
\nu ,\frac{1}{2}+\nu \\ 
1+\nu ,1+\nu ,1+2\nu%
\end{array}%
\right\vert z^{2}\right) \right.  \notag \\
&&\quad -\left. I_{\nu }\left( z\right) \Gamma ^{2}\left( \nu \right) \left( 
\frac{z}{2}\right) ^{-2\nu }\,_{2}F_{3}\left( \left. 
\begin{array}{c}
-\nu ,\frac{1}{2}-\nu \\ 
1-\nu ,1-\nu ,1-2\nu%
\end{array}%
\right\vert z^{2}\right) \right\} ,  \notag
\end{eqnarray}%
which is equivalent to the result given in \cite{BrychovNew}.

Finally, taking into account the main results of this Section, we can derive
an integral representation for the order derivative of the Macdonald
function $K_{\nu }\left( z\right) $.

\begin{theorem}
For $\nu >0$ and $z\neq 0$, $\left\vert \mathrm{arg}\ z\right\vert \leq \pi $%
, we have%
\begin{equation}
\frac{\partial K_{\nu }\left( z\right) }{\partial \nu }=2\nu \left[ K_{\nu
}\left( z\right) \int_{z}^{\infty }\frac{I_{\nu }\left( t\right) K_{\nu
}\left( t\right) }{t}dt-I_{\nu }\left( z\right) \int_{z}^{\infty }\frac{%
K_{\nu }^{2}\left( t\right) }{t}dt\right] .  \label{DKnu_int}
\end{equation}
\end{theorem}

\begin{proof}
Substituting (\ref{Int_Inu_Knu})\ and (\ref{Int_Knu^2})\ in (\ref{DKnu_int}%
), taking into account the definition given in (\ref{Knu_def})\ of the Macdonald
function $K_{\nu }\left( z\right) $, and the reflection formula (\ref%
{Gamma_reflection}), after some algebra we arrive at (\ref{DKnu_int}).
\end{proof}

\section{Alternative expressions for integral order\label{Section: Meijer_G}}

So far, we have obtained closed-form expressions for $\partial J_{\nu }/\partial \nu $ and $\partial Y_{\nu
}/\partial \nu $ in (\ref{DJnu_closed_form})\ and (\ref{DYnu_closed_form}),
and for $\partial I_{\nu }/\partial \nu $ and 
$\partial K_{\nu }/\partial \nu $ in (\ref{DInu_resultado})\ and (\ref%
{DKnu_resultado}). However, these expressions cannot be applied for $\nu \in 
\mathbb{Z}
$. Nonetheless, we can derive alternative expressions that avoid this problem using Meijer-$G$ functions. This function is usually defined by the following Mellin-Barnes integral representation \cite[Eqn. 16.17.1]{DLMF}:%
\begin{eqnarray}
&&G_{p,q}^{m,n}\left( z\left\vert 
\begin{array}{c}
a_{1},\ldots ,a_{p} \\ 
b_{1},\ldots ,b_{q}%
\end{array}%
\right. \right)  \label{Meijer_G_def} \\
&=&\frac{1}{2\pi i}\int_{L}\frac{\prod_{\ell =1}^{m}\Gamma \left( b_{\ell
}-s\right) \prod_{\ell =1}^{n}\Gamma \left( 1-a_{\ell }+s\right) }{%
\prod_{\ell =m}^{q-1}\Gamma \left( 1-b_{\ell +1}+s\right) \prod_{\ell
=n}^{p-1}\Gamma \left( a_{\ell +1}-s\right) }z^{s}ds,  \notag
\end{eqnarray}%
where the integration path $L$ separates the poles of the factors $\Gamma \left(
b_{\ell }-s\right) $ from those of the factors $\Gamma \left( 1-a_{\ell
}+s\right) $. Also, $m$ and $n$ are integers such that $0\leq m\leq q$ and $%
0\leq n\leq p$, and none of $a_{k}-b_{j}$ is a positive integer when $1\leq
k\leq n$ and $1\leq j\leq m$.

First, we introduce some properties of the Meijer-$G$ function that will be
used below. The Meijer-$G$ function satisfies the
following reduction formulas \cite[Eqns. 8.2.2(8)-(9)]{Prudnikov3}: 
\begin{equation}
G_{p,q}^{m,n}\left( z\left\vert 
\begin{array}{c}
a_{1},\ldots ,a_{p} \\ 
b_{1},\ldots ,b_{q-1},a_{1}%
\end{array}%
\right. \right) =G_{p-1,q-1}^{m,n-1}\left( z\left\vert 
\begin{array}{c}
a_{2},\ldots ,a_{p} \\ 
b_{1},\ldots ,b_{q-1}%
\end{array}%
\right. \right) ,  \label{Meijer_reduction_a1}
\end{equation}%
and%
\begin{equation}
G_{p,q}^{m,n}\left( z\left\vert 
\begin{array}{c}
a_{1},\ldots ,a_{p-1},b_{1} \\ 
b_{1},\ldots ,b_{q}%
\end{array}%
\right. \right) =G_{p-1,q-1}^{m-1,n}\left( z\left\vert 
\begin{array}{c}
a_{1},\ldots ,a_{p-1} \\ 
b_{2},\ldots ,b_{q}%
\end{array}%
\right. \right) .  \label{Meijer_reduction_ap}
\end{equation}

Also, it satisfies the following derivative formulas \cite[Eqns.
8.2.2(36)-(37)]{Prudnikov3}:%
\begin{eqnarray}
&&\frac{d}{dz}\left[ z^{1-a_{1}}\ G_{p,q}^{m,n}\left( z\left\vert 
\begin{array}{c}
a_{1},\ldots ,a_{p} \\ 
b_{1},\ldots ,b_{q}%
\end{array}%
\right. \right) \right]  \label{D_Meijer_a1} \\
&=&z^{-a_{1}}\ G_{p,q}^{m,n}\left( z\left\vert 
\begin{array}{c}
a_{1}-1,\ldots ,a_{p} \\ 
b_{1},\ldots ,b_{q}%
\end{array}%
\right. \right) ,\quad n\geq 1,  \notag
\end{eqnarray}%
and%
\begin{eqnarray}
&&\frac{d}{dz}\left[ z^{1-a_{p}}\ G_{p,q}^{m,n}\left( z\left\vert 
\begin{array}{c}
a_{1},\ldots ,a_{p} \\ 
b_{1},\ldots ,b_{q}%
\end{array}%
\right. \right) \right]  \label{D_Meijer_ap} \\
&=&-z^{-a_{p}}\ G_{p,q}^{m,n}\left( z\left\vert 
\begin{array}{c}
a_{1},\ldots ,a_{p}-1 \\ 
b_{1},\ldots ,b_{q}%
\end{array}%
\right. \right) ,\quad n\leq p-1.  \notag
\end{eqnarray}

The translation formula in the parameters reads as \cite[Eqn. 8.2.2(15)]%
{Prudnikov3},%
\begin{equation}
z^{\alpha }G_{p,q}^{m,n}\left( z\left\vert 
\begin{array}{c}
a_{1},\ldots ,a_{p} \\ 
b_{1},\ldots ,b_{q}%
\end{array}%
\right. \right) =G_{p,q}^{m,n}\left( z\left\vert 
\begin{array}{c}
a_{1}+\alpha ,\ldots ,a_{p}+\alpha \\ 
b_{1}+\alpha ,\ldots ,b_{q}+\alpha%
\end{array}%
\right. \right) .  \label{Meijer_translation}
\end{equation}

Also, the generalized hypergeometric function $_{p}F_{q}$\ can be expressed
in terms of the Meijer-$G$ function as follows \cite[Eqn. 8.4.51(1)]%
{Prudnikov3}: 
\begin{eqnarray}
&&_{p}F_{q}\left( \left. 
\begin{array}{c}
a_{1},\ldots ,a_{p} \\ 
b_{1},\ldots ,b_{q}%
\end{array}%
\right\vert -x\right)  \label{Meijer->p_F_q} \\
&=&\frac{\prod_{\ell =1}^{q}\Gamma \left( b_{\ell }\right) }{\prod_{\ell
=1}^{p}\Gamma \left( a_{\ell }\right) }\, G_{p,q+1}^{1,p}\left( x\left\vert 
\begin{array}{c}
1-a_{1},\ldots ,1-a_{p} \\ 
0,1-b_{1},\ldots ,1-b_{q}%
\end{array}%
\right. \right) .  \notag
\end{eqnarray}

Finally, for the asymptotic behavior of the Meijer-$G$ function (\ref%
{Meijer_G_def}), we introduce the following notation:%
\begin{eqnarray*}
\mu &=&q-m-n,\quad \sigma =q-p, \\
\Xi _{1} &=&\sum_{h=1}^{q}b_{h},\quad \Lambda _{1}=\sum_{h=1}^{p}a_{h}, \\
\theta &=&\frac{\left( 1-\sigma \right) /2+\Xi _{1}-\Lambda _{1}}{\sigma },
\\
A_{\quad q}^{m,n} &=&\left( -\frac{1}{2\pi i}\right) ^{\mu }\exp \left( i\pi %
\left[ \sum_{j=1}^{n}a_{j}-\sum_{j=m+1}^{q}b_{j}\right] \right) , \\
H_{p,q}\left( z\right) &=&\frac{\left( 2\pi \right) ^{\left( \sigma
-1\right) /2}}{\sigma ^{1/2}}\exp \left( -\sigma \,z^{1/\sigma }\right)
z^{\theta }\sum_{k=0}^{\infty }M_{k}\,z^{-k/\sigma },
\end{eqnarray*}%
where the first coefficient in the last expansion is $M_{0}=1$. Thereby,
according to \cite[Eqn. 5.10(8)]{Luke}, the following result is satisfied:

\begin{theorem}
If $0\leq n\leq p\leq q-2$, $p+1\leq m+n\leq \left( p+q\right) /2$, and $%
\mathrm{arg}\ z=0$, then%
\begin{equation}
G_{p,q}^{m,n}\left( z\right) \sim A_{\quad q}^{m,n}H_{p,q}\left( ze^{i\pi
\mu }\right) +\bar{A}_{\quad q}^{m,n}H_{p,q}\left( ze^{-i\pi \mu }\right)
,\quad z\rightarrow \infty .  \label{Meijer_G_z->inf}
\end{equation}
\end{theorem}

\subsection{Order derivatives of Bessel functions}

\begin{theorem}
$\forall \nu \in 
\mathbb{R}
$ and $\mathrm{Re\,}z>0$, the following integral holds true: 
\begin{equation}
\int_{z}^{\infty }\frac{J_{\nu }\left( t\right) Y_{\nu }\left( t\right) }{t}%
dt=\frac{-1}{2\sqrt{\pi }}\, G_{2,4}^{3,0}\left( z^{2}\left\vert 
\begin{array}{c}
1/2,1 \\ 
0,0,\nu ,-\nu%
\end{array}%
\right. \right) .  \label{Int_Jnu_Ynu_Meijer}
\end{equation}
\end{theorem}

\begin{proof}
According to the representation \cite[Eqn. 8.4.20(9)]{Prudnikov3}%
\begin{equation*}
J_{\nu }\left( \sqrt{x}\right) Y_{\nu }\left( \sqrt{x}\right) =-\frac{1}{%
\sqrt{\pi }}\, G_{1,3}^{2,0}\left( x\left\vert 
\begin{array}{c}
1/2 \\ 
0,\nu ,-\nu%
\end{array}%
\right. \right) ,
\end{equation*}%
we have the following indefinite integral,%
\begin{equation*}
\int \frac{J_{\nu }\left( t\right) Y_{\nu }\left( t\right) }{t}dt=-\frac{1}{%
\sqrt{\pi }}\int G_{1,3}^{2,0}\left( t^{2}\left\vert 
\begin{array}{c}
1/2 \\ 
0,\nu ,-\nu%
\end{array}%
\right. \right) \frac{dt}{t}.
\end{equation*}%
Performing the change of variables $u=t^{2}$ and applying the reduction
formula (\ref{Meijer_reduction_ap}), we obtain%
\begin{equation*}
\int \frac{J_{\nu }\left( t\right) Y_{\nu }\left( t\right) }{t}dt=-\frac{1}{2%
\sqrt{\pi }}\int G_{2,4}^{3,0}\left( u\left\vert 
\begin{array}{c}
1/2,0 \\ 
0,0,\nu ,-\nu%
\end{array}%
\right. \right) \frac{du}{u}.
\end{equation*}%
Taking now $a_{p}=1$ in the derivative formula (\ref{D_Meijer_ap}), we
arrive at%
\begin{equation*}
\int \frac{J_{\nu }\left( t\right) Y_{\nu }\left( t\right) }{t}dt=\frac{1}{2%
\sqrt{\pi }}\, G_{2,4}^{3,0}\left( t^{2}\left\vert 
\begin{array}{c}
1/2,1 \\ 
0,0,\nu ,-\nu%
\end{array}%
\right. \right) .
\end{equation*}%
Since, according to (\ref{Meijer_G_z->inf}), 
\begin{equation*}
\lim_{z\rightarrow \infty }\, G_{2,4}^{3,0}\left( z^{2}\left\vert 
\begin{array}{c}
1/2,1 \\ 
0,0,\nu ,-\nu%
\end{array}%
\right. \right) =\lim_{z\rightarrow \infty }\, \frac{\sin \pi \nu }{\sqrt{\pi }%
z^{2}}e^{-2iz}=0,
\end{equation*}%
we conclude (\ref{Int_Jnu_Ynu_Meijer}), as we wanted to prove.
\end{proof}

\begin{theorem}
$\forall \nu >0$ and $\mathrm{Re\,}z>0$, the following integrals holds true:%
\begin{eqnarray}
&&\int_{z}^{\infty }\frac{Y_{\nu }^{2}\left( t\right) }{t}dt
\label{Int_Y2_Meijer} \\
&=&\frac{1}{2\nu }+\frac{1}{\sqrt{\pi }}\, G_{3,5}^{4,0}\left( z^{2}\left\vert 
\begin{array}{c}
1/2,1/2-\nu ,1 \\ 
0,0,\nu ,-\nu ,1/2-\nu%
\end{array}%
\right. \right)  \notag \\
&&-\frac{\left( z/2\right) ^{2\nu }}{2\nu\, \Gamma ^{2}\left( \nu +1\right) }%
\,_{2}F_{3}\left( \left. 
\begin{array}{c}
\nu ,1/2+\nu \\ 
2\nu +1,\nu +1,\nu +1%
\end{array}%
\right\vert -z^{2}\right) .  \notag
\end{eqnarray}
\end{theorem}

\begin{proof}
According to the representation \cite[Eqn. 8.4.20(7)]{Prudnikov3}%
\begin{eqnarray*}
Y_{\nu }^{2}\left( \sqrt{x}\right) &=&\frac{2}{\sqrt{\pi }}\, %
G_{2,4}^{3,0}\left( x\left\vert 
\begin{array}{c}
1/2,1/2-\nu \\ 
0,\nu ,-\nu ,1/2-\nu%
\end{array}%
\right. \right) \\
&&+\frac{1}{\sqrt{\pi }}\, G_{1,3}^{1,1}\left( x\left\vert 
\begin{array}{c}
1/2 \\ 
\nu ,-\nu ,0%
\end{array}%
\right. \right) ,
\end{eqnarray*}%
we have the following indefinite integral,%
\begin{eqnarray}
&&\int \frac{Y_{\nu }^{2}\left( t\right) }{t}dt  \label{Int_Ynu^2_Meijer_1}
\\
&=&\frac{1}{\sqrt{\pi }}\int G_{2,4}^{3,0}\left( u\left\vert 
\begin{array}{c}
1/2,1/2-\nu \\ 
0,\nu ,-\nu ,1/2-\nu%
\end{array}%
\right. \right) \frac{du}{u}  \notag \\
&&+\frac{1}{2\sqrt{\pi }}\int G_{1,3}^{1,1}\left( u\left\vert 
\begin{array}{c}
1/2 \\ 
\nu ,-\nu ,0%
\end{array}%
\right. \right) \frac{du}{u},  \notag
\end{eqnarray}%
where we have performed the change of variables $u=t^{2}$. The first
integral on the RHS\ of (\ref{Int_Ynu^2_Meijer_1})\ is calculated using the
reduction formula (\ref{Meijer_reduction_ap}),\ and then applying the derivative
formula (\ref{D_Meijer_ap}) with $a_{p}=1$, 
\begin{eqnarray}
&&\frac{1}{\sqrt{\pi }}\int G_{2,4}^{3,0}\left( u\left\vert 
\begin{array}{c}
1/2,1/2-\nu \\ 
0,\nu ,-\nu ,1/2-\nu%
\end{array}%
\right. \right) \frac{du}{u}  \label{Int_Meijer_1} \\
&=&\frac{-1}{\sqrt{\pi }}\, G_{3,5}^{4,0}\left( t^{2}\left\vert 
\begin{array}{c}
1/2,1/2-\nu ,1 \\ 
0,0,\nu ,-\nu ,1/2-\nu%
\end{array}%
\right. \right) .  \notag
\end{eqnarray}%
The second integral on the RHS\ of (\ref{Int_Ynu^2_Meijer_1})\ is calculated
using the reduction formula (\ref{Meijer_reduction_a1}),\ and then applying the
derivative formula (\ref{D_Meijer_a1}) with $a_{1}=1$, 
\begin{equation}
\frac{1}{2\sqrt{\pi }}\int G_{1,3}^{1,1}\left( u\left\vert 
\begin{array}{c}
1/2 \\ 
\nu ,-\nu ,0%
\end{array}%
\right. \right) \frac{du}{u}=\frac{1}{2\sqrt{\pi }}\, G_{2,4}^{1,2}\left(
t^{2}\left\vert 
\begin{array}{c}
1,1/2 \\ 
\nu ,-\nu ,0,0%
\end{array}%
\right. \right) .  \label{Meijer->Hypergeometric}
\end{equation}%
Notice that applying the translation formula (\ref{Meijer_translation}) and
then the reduction formula (\ref{Meijer->p_F_q}), we can express the RHS\ of
(\ref{Meijer->Hypergeometric})\ as a hypergeometric function, 
\begin{eqnarray}
&&\frac{1}{2\sqrt{\pi }}\int G_{1,3}^{1,1}\left( u\left\vert 
\begin{array}{c}
1/2 \\ 
\nu ,-\nu ,0%
\end{array}%
\right. \right) \frac{du}{u}  \label{Int_Meijer_2} \\
&=&\frac{t^{2\nu }}{2\sqrt{\pi }}\, G_{2,4}^{1,2}\left( t^{2}\left\vert 
\begin{array}{c}
1-\nu ,1/2-\nu \\ 
0,-2\nu ,-\nu ,-\nu%
\end{array}%
\right. \right)  \notag \\
&=&\frac{\left( t/2\right) ^{2\nu }}{2\,\Gamma ^{2}\left( \nu +1\right) }%
\,_{2}F_{3}\left( \left. 
\begin{array}{c}
\nu ,1/2+\nu \\ 
2\nu +1,\nu +1,\nu +1%
\end{array}%
\right\vert -t^{2}\right) ,  \notag
\end{eqnarray}%
where we have applied the duplication formula of the gamma function (\ref%
{Gamma_duplication}). Therefore, substituting the results (\ref{Int_Meijer_1}%
)\ and (\ref{Int_Meijer_2})\ in (\ref{Int_Ynu^2_Meijer_1}), we obtain%
\begin{eqnarray}
&&\int \frac{Y_{\nu }^{2}\left( t\right) }{t}dt
\label{Int_Y2_Meijer_indefinida} \\
&=&\frac{-1}{\sqrt{\pi }}\, G_{3,5}^{4,0}\left( t^{2}\left\vert 
\begin{array}{c}
1/2,1/2-\nu ,1 \\ 
0,0,\nu ,-\nu ,1/2-\nu%
\end{array}%
\right. \right)  \notag \\
&&+\frac{\left( t/2\right) ^{2\nu }}{2\,\Gamma ^{2}\left( \nu +1\right) }%
\,_{2}F_{3}\left( \left. 
\begin{array}{c}
\nu ,1/2+\nu \\ 
2\nu +1,\nu +1,\nu +1%
\end{array}%
\right\vert -t^{2}\right)  \notag
\end{eqnarray}%
Finally, notice that the hypergeometric function obtained in (\ref%
{Int_Meijer_2})\ is the integral calculated in (\ref{Int_I2}), thus%
\begin{eqnarray}
&&\lim_{z\rightarrow \infty }\, \frac{\left( z/2\right) ^{2\nu }}{2\,\Gamma
^{2}\left( \nu +1\right) }\,_{2}F_{3}\left( \left. 
\begin{array}{c}
\nu ,1/2+\nu \\ 
2\nu +1,\nu +1,\nu +1%
\end{array}%
\right\vert -z^{2}\right)  \label{Lim_Meijer_1} \\
&=&\int_{0}^{\infty }\frac{J_{\nu }^{2}\left( t\right) }{t}dt=\frac{1}{2\nu }%
,  \notag
\end{eqnarray}%
where we have applied \cite[Eqn. 13.42(1)]{Watson}. Also, according to (\ref%
{Meijer_G_z->inf}), we have that 
\begin{equation}
\lim_{z\rightarrow \infty }\, G_{3,5}^{4,0}\left( z^{2}\left\vert 
\begin{array}{c}
1/2,1/2-\nu ,1 \\ 
0,0,\nu ,-\nu ,1/2-\nu%
\end{array}%
\right. \right) =\lim_{z\rightarrow \infty }\, \frac{-\cos \pi \nu }{\sqrt{\pi }%
z^{2}}e^{-2iz}=0.  \label{Lim_Meijer_2}
\end{equation}%
Therefore, from (\ref{Int_Y2_Meijer_indefinida}), (\ref{Lim_Meijer_1})\ and (%
\ref{Lim_Meijer_2}), we conclude (\ref{Int_Y2_Meijer}), as we wanted to
prove.
\end{proof}

According to the integral representation given in (\ref{DJnu_int_Dunster}),
and the results obtained in (\ref{Int_Jnu^2b})\ and (\ref{Int_Jnu_Ynu_Meijer}%
), we calculate the order derivative of the Bessel function of the first
kind as,%
\begin{eqnarray}
\frac{\partial J_{\nu }\left( z\right) }{\partial \nu } &=&\frac{\pi }{2}%
\left[ Y_{\nu }\left( z\right) \frac{\left( z/2\right) ^{2\nu }}{\,\Gamma
^{2}\left( \nu +1\right) }\,_{2}F_{3}\left( \left. 
\begin{array}{c}
\nu ,1/2+\nu \\ 
2\nu +1,\nu +1,\nu +1%
\end{array}%
\right\vert -z^{2}\right) \right.  \label{DJnu_Meijer} \\
&&-\left. \frac{\nu J_{\nu }\left( z\right) }{\sqrt{\pi }}\, %
G_{2,4}^{3,0}\left( z^{2}\left\vert 
\begin{array}{c}
1/2,1 \\ 
0,0,\nu ,-\nu%
\end{array}%
\right. \right) \right] ,\quad \nu >0,\,\mathrm{Re\,}z>0.  \notag
\end{eqnarray}

As by-product, from (\ref{DJnu_int_Apelblat}) and (\ref{DJnu_Meijer}), we
obtain the calculation of the following integral, which does not seem to be
reported in the literature,%
\begin{eqnarray}
&&\int_{0}^{\pi /2}\tan \theta \ Y_{0}\left( z\sin ^{2}\theta \right) J_{\nu
}\left( z\cos ^{2}\theta \right) d\theta  \label{Int_Jnu_Appleblat_resultado}
\\
&=&\frac{Y_{\nu }\left( z\right) \left( t/2\right) ^{2\nu }}{\,2\nu\, \Gamma
^{2}\left( \nu +1\right) }\,_{2}F_{3}\left( \left. 
\begin{array}{c}
\nu ,1/2+\nu \\ 
2\nu +1,\nu +1,\nu +1%
\end{array}%
\right\vert -t^{2}\right)  \notag \\
&&-\frac{J_{\nu }\left( z\right) }{2\sqrt{\pi }}\, G_{2,4}^{3,0}\left(
z^{2}\left\vert 
\begin{array}{c}
1/2,1 \\ 
0,0,\nu ,-\nu%
\end{array}%
\right. \right) ,\quad \,\nu >0.  \notag
\end{eqnarray}

Also, according to the integral representation given in (\ref%
{DYnu_int_Dunster}), and the results obtained in (\ref{Int_Jnu_Ynu_Meijer})
and (\ref{Int_Y2_Meijer}), the order derivative of the Bessel function of
the second kind is%
\begin{eqnarray}
\frac{\partial Y_{\nu }\left( z\right) }{\partial \nu } &=&J_{\nu }\left(
z\right) \left[ \sqrt{\pi }\nu\ G_{3,5}^{4,0}\left( z^{2}\left\vert 
\begin{array}{c}
1/2,1/2-\nu ,1 \\ 
0,0,\nu ,-\nu ,1/2-\nu%
\end{array}%
\right. \right) \right.  \label{DYnu_Meijer} \\
&&\left. -\frac{\pi \left( z/2\right) ^{2\nu }}{2\, \Gamma ^{2}\left( \nu
+1\right) }\,_{2}F_{3}\left( \left. 
\begin{array}{c}
\nu ,1/2+\nu \\ 
2\nu +1,\nu +1,\nu +1%
\end{array}%
\right\vert -z^{2}\right) \right]  \notag \\
&&+\frac{\sqrt{\pi }\nu\ Y_{\nu }\left( z\right) } {2}\, G_{2,4}^{3,0}\left(
z^{2}\left\vert 
\begin{array}{c}
1/2,1 \\ 
0,0,\nu ,-\nu%
\end{array}%
\right. \right) ,\quad \nu >0,\,\mathrm{Re\,}z>0.  \notag
\end{eqnarray}

\subsection{Order derivatives of modified Bessel functions}

\begin{theorem}
$\forall \nu >0$ and $\mathrm{Re\,}z>0$, the following integral holds true:\ 
\begin{equation}
\int_{z}^{\infty }\frac{I_{\nu }\left( t\right) K_{\nu }\left( t\right) }{t}%
dt=\frac{1}{4\sqrt{\pi }}\, G_{2,4}^{3,1}\left( z^{2}\left\vert 
\begin{array}{c}
1/2,1 \\ 
0,0,\nu ,-\nu%
\end{array}%
\right. \right) .  \label{Int_IK_Meijer}
\end{equation}
\end{theorem}

\begin{proof}
According to the representation \cite[Eqn. 8.4.23(19)]{Prudnikov3}%
\begin{equation*}
I_{\nu }\left( \sqrt{x}\right) K_{\nu }\left( \sqrt{x}\right) =\frac{1}{2%
\sqrt{\pi }}\, G_{1,3}^{2,1}\left( x\left\vert 
\begin{array}{c}
1/2 \\ 
0,\nu ,-\nu%
\end{array}%
\right. \right) ,
\end{equation*}%
we have the following indefinite integral%
\begin{equation*}
\int \frac{I_{\nu }\left( t\right) K_{\nu }\left( t\right) }{t}dt=\frac{1}{4%
\sqrt{\pi }}\int G_{1,3}^{2,1}\left( u\left\vert 
\begin{array}{c}
1/2 \\ 
0,\nu ,-\nu%
\end{array}%
\right. \right) \frac{du}{u},
\end{equation*}%
where we have performed the change of variables $u=t^{2}$. Now, applying the
reduction formula (\ref{Meijer_reduction_ap})\ and the derivative formula (%
\ref{D_Meijer_ap}) with $a_{p}=1$, we arrive at%
\begin{equation}
\int \frac{I_{\nu }\left( t\right) K_{\nu }\left( t\right) }{t}dt=-\frac{1}{4%
\sqrt{\pi }}\, G_{2,4}^{3,1}\left( t^{2}\left\vert 
\begin{array}{c}
1/2,1 \\ 
0,0,\nu ,-\nu%
\end{array}%
\right. \right) .  \label{Int_IK_Meijer_a}
\end{equation}%
Finally, note that, according to (\ref{Meijer_G_z->inf}), we have%
\begin{equation*}
\lim_{z\rightarrow \infty }\, G_{2,4}^{3,1}\left( z^{2}\left\vert 
\begin{array}{c}
1/2,1 \\ 
0,0,\nu ,-\nu%
\end{array}%
\right. \right) =-2\sqrt{\pi }\sin \pi \nu \lim_{z\rightarrow \infty }\, \frac{%
e^{-2z}}{z^{2}}=0,
\end{equation*}%
thus we get (\ref{Int_IK_Meijer})\ from (\ref{Int_IK_Meijer_a}), as we
wanted to prove.
\end{proof}

\begin{theorem}
$\forall \nu \in 
\mathbb{R}
$ and $\mathrm{Re\,}z>0$, the following integral holds true:%
\begin{equation}
\int_{z}^{\infty }\frac{K_{\nu }^{2}\left( t\right) }{t}dt=\frac{\sqrt{\pi }%
}{4}\, G_{2,4}^{4,0}\left( z^{2}\left\vert 
\begin{array}{c}
1/2,1 \\ 
0,0,\nu ,-\nu%
\end{array}%
\right. \right) .  \label{Int_K2_Meijer}
\end{equation}
\end{theorem}

\begin{proof}
Following the same steps as in the previous theorem, departing from the
representation \cite[Eqn. 8.4.23(27)]{Prudnikov3}%
\begin{equation*}
K_{\nu }^{2}\left( \sqrt{x}\right) =\frac{\sqrt{\pi }}{2}\, G_{1,3}^{3,0}\left(
x\left\vert 
\begin{array}{c}
1/2 \\ 
0,\nu ,-\nu%
\end{array}%
\right. \right) ,
\end{equation*}%
we obtain the desired result.
\end{proof}

According to the integral representation given in (\ref{DInu_int}), and the
results obtained in (\ref{Int_I2})\ and (\ref{Int_IK_Meijer}), we calculate
the order derivative of the modified Bessel function as,%
\begin{eqnarray}
\frac{\partial I_{\nu }\left( z\right) }{\partial \nu } &=&\frac{-\nu I_{\nu
}\left( z\right) }{2\sqrt{\pi }}\, G_{2,4}^{3,1}\left( z^{2}\left\vert 
\begin{array}{c}
1/2,1 \\ 
0,0,\nu ,-\nu%
\end{array}%
\right. \right)  \label{DInu_Meijer} \\
&&-\frac{K_{\nu }\left( z\right) \left( z/2\right) ^{2\nu }}{\Gamma
^{2}\left( \nu +1\right) }\,_{2}F_{3}\left( \left. 
\begin{array}{c}
\nu ,\nu +\frac{1}{2} \\ 
\nu +1,\nu +1,2\nu +1%
\end{array}%
\right\vert z^{2}\right) ,  \notag \\
\quad \nu &>&0,\,\mathrm{Re\,}z>0.  \notag
\end{eqnarray}

As by-product, according to (\ref{DInu_int_Apelblat})\ and (\ref{DInu_Meijer}%
), we calculate the following integral, which does not seem to be reported
in the literature,%
\begin{eqnarray}
&&\int_{0}^{\pi /2}\tan \theta \ K_{0}\left( z\sin ^{2}\theta \right) I_{\nu
}\left( z\cos ^{2}\theta \right) d\theta  \label{Int_Inu_Appleblat_resultado}
\\
&=&\frac{I_{\nu }\left( z\right) }{4\sqrt{\pi }}\, G_{2,4}^{3,1}\left(
z^{2}\left\vert 
\begin{array}{c}
1/2,1 \\ 
0,0,\nu ,-\nu%
\end{array}%
\right. \right)  \notag \\
&&+\frac{K_{\nu }\left( z\right) \left( z/2\right) ^{2\nu }}{2\nu \, \Gamma
^{2}\left( \nu +1\right) }\,_{2}F_{3}\left( \left. 
\begin{array}{c}
\nu ,\nu +\frac{1}{2} \\ 
\nu +1,\nu +1,2\nu +1%
\end{array}%
\right\vert z^{2}\right) ,\quad \nu >0.  \notag
\end{eqnarray}

Finally, according to the integral representation given in (\ref{DKnu_int}),
and the results obtained in (\ref{Int_IK_Meijer}) and (\ref{Int_K2_Meijer}),
the order derivative of the Macdonald function is%
\begin{eqnarray}
&&\frac{\partial K_{\nu }\left( z\right) }{\partial \nu }
\label{DKnu_Meijer} \\
= &&\frac{\nu }{2}\left[ \frac{K_{\nu }\left( z\right) }{\sqrt{\pi }}\, %
G_{2,4}^{3,1}\left( z^{2}\left\vert 
\begin{array}{c}
1/2,1 \\ 
0,0,\nu ,-\nu%
\end{array}%
\right. \right) -\sqrt{\pi }I_{\nu }\left( z\right) G_{2,4}^{4,0}\left(
z^{2}\left\vert 
\begin{array}{c}
1/2,1 \\ 
0,0,\nu ,-\nu%
\end{array}%
\right. \right) \right] ,  \notag \\
\quad &&\nu >0,\,\mathrm{Re\,}z>0.  \notag
\end{eqnarray}

\section{Conclusions\label{Section: Conclusions}}

We have calculated some integrals involving Bessel functions in terms of
generalized hypergeometric functions in (\ref{Int_Jnu^2b}), (\ref{Int_Jnu_Ynu}%
), and (\ref{Int_Y2}); and in terms of Meijer-$G$ functions in (\ref%
{Int_Jnu_Ynu_Meijer}) and (\ref{Int_Y2_Meijer}). These integrals have been
applied to express in closed-form the derivative of the Bessel
functions with respect to the order from integral representations given in the literature, i.e. (\ref%
{DJnu_int_Dunster})\ and (\ref{DYnu_int_Dunster}). We have expressed these results using hypergeometric functions, namely (\ref{DJnu_closed_form}) and (\ref%
{DYnu_closed_form}), and Meijer-$G$
functions, namely (\ref{DJnu_Meijer})\ and (\ref{DYnu_Meijer}). We have carried out similar
calculations to obtain closed-form expressions for the derivative of the modified Bessel functions with respect to the order, both in terms of hypergeometric functions, namely (\ref{DInu_resultado}) and (\ref{DKnu_resultado}),\ as well as
in terms of Meijer-$G$\ functions, namely (\ref{DInu_Meijer}) and (\ref%
{DKnu_Meijer}). For this purpose, we have obtained integral representations for the order
derivative of the modified Bessel functions in (\ref%
{DInu_int})\ and (\ref{DKnu_int}).

Despite the fact, the expressions given in (\ref{DJnu_closed_form}) and (\ref%
{DYnu_closed_form}) for $\partial J_{\nu }/\partial \nu $ and $\partial
Y_{\nu }/\partial \nu $\ as well as the expressions\ (\ref{DInu_resultado})
and (\ref{DKnu_resultado}) for $\partial I_{\nu }/\partial \nu $ and $%
\partial K_{\nu }/\partial \nu $ have been already calculated in \cite%
{BrychovNew} with the aid of symbolic computer algebra, we have presented
here a formal derivation which turns out to be highly non-trivial. However,
we cannot use these expressions for integral order $\nu $. Nonetheless, this
is not the case for the expressions using Meijer-$G$\ functions,
namely (\ref{DJnu_Meijer}), (\ref{DYnu_Meijer}), (\ref{DInu_Meijer}) and (%
\ref{DKnu_Meijer}). However, for non-integral order $\nu $, the former
expressions are computed much more rapidly than these latter ones ($\approx 10$
times faster).

Finally, as by-products, we have calculated two integrals in (\ref%
{Int_Jnu_Appleblat_resultado})\ and (\ref{Int_Inu_Appleblat_resultado}),
which do not seem to be reported in the literature.





\end{document}